\documentclass[12pt,oneside,english]{amsart}
\usepackage[T1]{fontenc}
\usepackage[latin9]{inputenc}
\usepackage{geometry}
\usepackage{color}
\usepackage{babel}
\usepackage{mathtools}
\usepackage{amstext}
\usepackage{amsthm}
\usepackage{amssymb}
\usepackage{esint}
\usepackage[unicode=true,pdfusetitle,
 bookmarks=true,bookmarksnumbered=false,bookmarksopen=false,
 breaklinks=false,pdfborder={0 0 0},pdfborderstyle={},backref=false,colorlinks=true]
 {hyperref}

\makeatletter
\numberwithin{equation}{section}
\numberwithin{figure}{section}
\theoremstyle{plain}
\newtheorem{thm}{\protect\theoremname}[section]
\theoremstyle{remark}
\newtheorem{rem}{\protect\remarkname}[section]
\theoremstyle{plain}
\newtheorem{cor}[thm]{\protect\corollaryname}
\theoremstyle{definition}
\newtheorem{defn}[thm]{\protect\definitionname}
\theoremstyle{plain}
\newtheorem{lem}[thm]{\protect\lemmaname}
\theoremstyle{plain}
\newtheorem{prop}[thm]{\protect\propositionname}
\theoremstyle{remark}
\newtheorem{claim}[thm]{\protect\claimname}

\makeatother

\providecommand{\claimname}{Claim}
\providecommand{\corollaryname}{Corollary}
\providecommand{\definitionname}{Definition}
\providecommand{\lemmaname}{Lemma}
\providecommand{\propositionname}{Proposition}
\providecommand{\remarkname}{Remark}
\providecommand{\theoremname}{Theorem}

\begin{document}
\global\long\def\cI{\mathcal{I}}

\global\long\def\cJ{\mathcal{J}}

\global\long\def\cN{\mathcal{N}}

\global\long\def\cX{\mathbb{\mathcal{X}}}

\global\long\def\cZ{\mathbb{\mathcal{Z}}}

\global\long\def\bbc{\mathbb{C}}

\global\long\def\bbd{\mathbb{D}}

\global\long\def\bbe{\mathbb{E}}

\global\long\def\bbn{\mathbb{N}}

\global\long\def\bbp{\mathbb{P}}

\global\long\def\bbr{\mathbb{R}}

\global\long\def\bbt{\mathbb{T}}

\global\long\def\bbz{\mathbb{Z}}

\global\long\def\varint#1#2#3{\cI\left(#1,#2,#3\right)}

\global\long\def\pr#1{\mathbb{\bbp}\left[#1\right]}

\global\long\def\ex#1{\mathbb{\bbe}\left[#1\right]}

\global\long\def\var#1{\mathrm{Var}\left(#1\right)}

\global\long\def\im#1{\mathrm{Im}\left\{  #1\right\}  }

\global\long\def\zs#1{\cZ_{#1}}

\global\long\def\eqdef{\coloneqq}

\global\long\def\defeq{\eqqcolon}

\global\long\def\dd{\mathrm{d}}

\global\long\def\veps{\varepsilon}

\global\long\def\indf#1#2{\mathbf{1}_{#2}\left(#1\right)}

\global\long\def\diskrc#1{#1\overline{\bbd}}

\global\long\def\Re#1{\mathrm{Re}\left[#1\right]}

\global\long\def\Im#1{\mathrm{Im}\left[#1\right]}

\global\long\def\longint#1{\mathring{T}_{#1}}

\title{Fluctuations for zeros of Gaussian Taylor series }

\author{Avner Kiro\textsuperscript{1} and Alon Nishry\textsuperscript{2}}
\footnotetext[1]{Weizmann Institute of Science, Rehovot, Israel. Email: avner-ephraiem.kiro@weizmann.ac.il. Supported in part by ERC Advanced Grant 692616, ISF Grant 382/15, and by the European Research Council (ERC) under the European Union's Horizon 2020 research and innovation programme (grant agreement No 802107).}
\footnotetext[2]{School of Mathematical Sciences, Tel Aviv
	University, Tel Aviv, Israel. Email: alonish@tauex.tau.ac.il. Supported in part by  ISF Grant 1903/18, and  by a grant from the United States - Israel Binational Science Foundation (BSF Start up Grant no. 2018341),
	Jerusalem, Israel.}
\begin{abstract}
We study fluctuations in the number of zeros of random analytic functions
given by a Taylor series whose coefficients are independent complex
Gaussians. When the functions are entire, we find sharp bounds for
the asymptotic growth rate of the variance of the number of zeros
in large disks centered at the origin. To obtain a result that holds
under no assumptions on the variance of the Taylor coefficients we
employ the Wiman-Valiron theory. We demonstrate the sharpness of our
bounds by studying well-behaved covariance kernels, which we call
admissible (after Hayman). 
\end{abstract}

\maketitle

\section{Introduction}

Some of the earliest works concerning random analytic functions are
the ones of Littlewood and Offord \cite{Littlewood1945distribution1,Littlewood1948distribution2},
who showed that the structure of the zero set of these functions is
very regular. More recently, these classical results were sharpened
in the papers \cite{Kabluchko2014distribution,nns2016distribution}.
In this paper we consider the typical size of fluctuations in the
number of zeros of random analytic functions whose coefficients are
independent complex Gaussians. This is the most well-studied and best
understood model (see the book \cite{ZerosBook} and ICM notes \cite{nazarov2010random}).

Given a sequence $\left\{ a_{n}\right\} _{n\geq0}$ of non-negative
numbers, we consider random Taylor series
\begin{equation}
f\left(z\right)=\sum_{n\geq0}\xi_{n}a_{n}z^{n},\label{eq:GAF_def}
\end{equation}
where $\xi_{n}$ are independent and identically distributed standard
complex Gaussians. We only consider \emph{transcendental} analytic
functions, that is, sequences $\left\{ a_{n}\right\} $ which contain
infinitely many non-zero terms. Denote by $\zs f=f^{-1}\left\{ 0\right\} $
the zero set of $f$; its properties are determined by the covariance
kernel
\[
K\left(z,w\right)=\ex{f\left(z\right)\overline{f\left(w\right)}}\defeq G(z\bar{w}),\quad\text{where}\quad G(z):=\sum_{n\geq0}a_{n}^{2}z^{n}.
\]
We will call $G$ the \emph{covariance function} of $f$; denote by
$R_{G}$ the radius of convergence of $G$ around the origin. We consider
both $R_{G}<\infty$ and $R_{G}=\infty$, and in the former case,
without loss of generality, we may assume $R_{G}=1$. Then, it is
not difficult to check that the radius of convergence of $f$ is almost
surely $R_{G}$ in both cases (see \cite[Lemma 2.2.3]{ZerosBook}).
When $R_{G}=\infty$ we call $f$ a \emph{Gaussian} \emph{entire}
\emph{function}.

Let $n_{f}\left(r\right)$ be the number of zeros of the function
$f$ inside the disk $\left\{ \left|z\right|\le r\right\} $, where
$r<R_{G}$. We are interested in the asymptotic statistical properties
of the random variable $n_{f}\left(r\right)$, as $r\to R_{G}$. In
order to study this asymptotics, it will be convenient to define the
following functions
\[
a\left(r\right)=a_{G}\left(r\right)\eqdef r\left(\log G\left(r\right)\right)^{\prime}=\frac{rG^{\prime}\left(r\right)}{G\left(r\right)},\qquad b\left(r\right)=b_{G}\left(r\right)\eqdef ra^{\prime}\left(r\right),
\]
borrowing the notation used in \cite{Hayman}. Since the Taylor coefficients
of $G$ are non-negative, the function $r\mapsto\log G\left(e^{r}\right)$
is \emph{convex}, hence $a\left(r\right)$ is increasing, and $b\left(r\right)$
is non-negative for all $r<R_{G}$.

The Edelman-Kostlan formula \cite[p. 25]{ZerosBook} (see also Appendix
\ref{sec:Kahane_formulas}) states that for \emph{any} Gaussian analytic
function $f$ of the form (\ref{eq:GAF_def}) we have
\[
\ex{n_{f}\left(r\right)}=a\left(r^{2}\right),\qquad\text{for all }r<R_{G}.
\]
However, the expected value provides little information about the
distribution of the random variable. Here we will be interested in
the asymptotic growth rate of the variance $\var{n_{f}\left(r\right)}$
in terms of the functions $a$ and $b$, in general under no additional
assumptions on the nature of the coefficients $a_{n}$.

In order to present the results we will need the following notation.
We say that a set $E\subset\bbr^{+}$ is of \emph{finite logarithmic
measure} if
\[
\int_{E}\,\frac{\dd t}{t}<\infty.
\]
If $g_{1},g_{2}:\bbr^{+}\to\bbr^{+}$ are non-negative functions,
we write $g_{1}\lesssim_{L}g_{2}$ if there is a constant $C>0$,
and a set $E\subset\bbr^{+}$ of finite logarithmic measure, such
that $g_{1}\le Cg_{2}$ in $\bbr^{+}\backslash E$. Finally, we write
$g_{1}\asymp_{L}g_{2}$ if $g_{1}\lesssim_{L}g_{2}$ and $g_{1}\gtrsim_{L}g_{2}$
both hold.

\begin{thm}
\label{thm:lower_bound} For \emph{any} Gaussian entire function $f$
with a transcendental covariance function $G$, and any $\varepsilon>0$
\[
\text{Var}\left(n_{f}\left(r\right)\right)\gtrsim_{L}\frac{b^{2}\left(r^{2}\right)}{a\left(r^{2}\right)^{\frac{3}{2}+\varepsilon}}.
\]
In addition, if \textbf{$b$} is a \emph{non-decreasing} function,
then
\[
\text{Var}\left(n_{f}\left(r\right)\right)\gtrsim_{L}\sqrt{b\left(r^{2}\right)}.
\]
\end{thm}

\begin{rem}
With some more work the factor $a\left(r^{2}\right)^{\veps}$ in Theorem
\ref{thm:lower_bound} can be replaced by a power of $\log a\left(r^{2}\right)$,
we will not pursue this here.
\end{rem}

\begin{rem}
\label{rmk:rosenbloom_a_b} By \cite[Lemma 1]{rosenbloom1962probability}
it follows that for every $\varepsilon>0$ we have that $b\left(r\right)\lesssim_{L}a\left(r\right)^{1+\varepsilon}$
.
\end{rem}

It turns out that the upper bound for the variance may be considerably
larger asymptotically. The following result holds without \emph{any}
restrictions on $G$.
\begin{thm}
\label{thm:upper_bound} Let $f$ be a Gaussian analytic function
with covariance function $G$, then for every $r<R_{G}$
\[
\text{Var}\left(n_{f}(r)\right)\le b\left(r^{2}\right).
\]
\end{thm}

\begin{rem}
\label{rmk:sharpness} If $G\left(z\right)$ is of the form $a_{n}^{2}z^{n}+a_{m}^{2}z^{m}$,
then one can check that $\var{n_{f}\left(r\right)}=b_{G}\left(r^{2}\right)$.
This implies that for any non-decreasing and unbounded function $\beta:\bbr^{+}\to\bbr^{+}$,
there is a covariance function $G_{\beta}$, so that if $f$ is the
Gaussian entire function whose covariance function is $G_{\beta}$,
there is a sequence $r_{n}\to\infty$ so that
\[
\var{n_{f}\left(r_{n}\right)}=\left(1+o\left(1\right)\right)\beta\left(r_{n}\right)=\left(1+o\left(1\right)\right)b_{G_{\beta}}\left(r_{n}^{2}\right),\qquad n\to\infty,
\]
and in particular
\[
\limsup_{r\to\infty}\frac{\var{n_{f}\left(r\right)}}{b_{G_{\beta}}\left(r^{2}\right)}=1.
\]
\end{rem}

By Theorem \ref{thm:upper_bound} and Remark \ref{rmk:rosenbloom_a_b}
we get the following conclusion.
\begin{cor}
\label{cor:var_bound_by_exp} For any Gaussian entire function $f$
with a transcendental covariance function $G$, and every $\varepsilon>0$,
\[
\var{n_{f}\left(r\right)}\lesssim_{L}\ex{n_{f}\left(r\right)}^{1+\veps}.
\]
\end{cor}

\subsection{Well-behaved covariance functions}

If the covariance function $G$ of $f$ is sufficiently well-behaved,
such as $e^{z}$, $e^{e^{z}}$, and the Mitag-Leffler functions
\[
G\left(z\right)=\sum_{n=0}^{\infty}\frac{z^{n}}{\Gamma\left(\alpha n+1\right)},
\]
then it is possible to find the asymptotics of the variance. A notable
example is the Gaussian Entire Function (GEF), with $G\left(z\right)=e^{z}$,
whose zero set is \emph{invariant} with respect to the isometries
of the complex plane (see the book \cite[Chapter 2.3]{ZerosBook}).
Forrester and Honner \cite{forrester1999exact} found the precise
asymptotic growth of the variance for the GEF (see also \cite{nazarov2011fluctuations}).
In order to extend this result, we define two classes of \emph{admissible}
covariance functions, which in particular include all the previous
examples. For the precise definitions see Sections \ref{subsec: type I admissible}
and \ref{subsec:type_II_admissible}. More examples, including Gaussian
analytic functions with an admissible covariance function in the unit
disk are described in Section \ref{sec:Examples}.
\begin{thm}
\label{thm:var_asymp} Let $f$ be a Gaussian analytic function with
a type I admissible covariance function $G$. Then 
\[
\var{n_{f}\left(r\right)}=\left(1+o\left(1\right)\right)\frac{\zeta\left(\frac{3}{2}\right)}{4\sqrt{\pi}}\sqrt{b\left(r^{2}\right)},\quad r\to R_{G}^{-},
\]
where $\zeta\left(u\right)$ is the Riemann zeta function.
\end{thm}

\begin{rem}
This indicates that the lower bound in Theorem \ref{thm:lower_bound}
is sharp.
\end{rem}

Suppose $G$ is a sufficiently regular covariance function (see Section
(\ref{subsec:type_II_admissible}) for the precise requirements).
In the next theorem we construct a Gaussian entire function $\widetilde{f}$
with covariance function $\widetilde{G}$, so that the variance of
the number of zeros of $\widetilde{f}$ is large outside a small exceptional
set of values of $r$. The statement of the theorem requires the following
definitions.
\begin{defn}
We will say that two covariance kernels $G$ and $\widetilde{G}$
are \emph{similar} if
\[
a_{\widetilde{G}}\left(r\right)\asymp_{L}a_{G}\left(r\right)\text{ and }b_{\widetilde{G}}\left(r\right)\asymp_{L}b_{G}\left(r\right).
\]
\end{defn}

\begin{defn}
\label{def:restriction} Let $G\left(z\right)=\sum_{n=0}^{\infty}c_{n}z^{n}$
be an analytic function. A function $\widetilde{G}$ is a \emph{Taylor
series restriction} of $G$, if $\widetilde{G}\left(z\right)=\sum_{n=0}^{\infty}\delta_{n}c_{n}z^{n}$
with $\delta_{n}\in\left\{ 0,1\right\} $ for all $n\in\bbn$.
\end{defn}

\begin{thm}
\label{thm:sharpness_upper_bound_example} Let $G$ be a type II admissible
function. There exists $\widetilde{G}$ which is a Taylor series restriction
of and similar to $G$, so that
\[
\var{n_{\widetilde{f}}\left(r\right)}\asymp_{L}b_{\widetilde{G}}\left(r^{2}\right).
\]
\end{thm}

\begin{rem}
The theorem shows that the upper bound in Theorem \ref{thm:upper_bound}
is sharp (up to a constant) for certain \emph{transcendental} entire
functions \emph{outside} a set of finite logarithmic measure (cf.
Remark \ref{rmk:sharpness}).
\end{rem}

\begin{rem}
By the example in Section \ref{subsec:double_exp_example} (which
is type II admissible) it follows that in general $\varepsilon$ \emph{cannot}
be removed in Corollary \ref{cor:var_bound_by_exp}.
\end{rem}

\subsection{Background and related results}

Following earlier work by Edelman and Kostlan \cite{edelman1995many},
and Offord \cite{offord1967distribution}, some fundamental properties
of zeros of Gaussian analytic functions (GAFs) were developed by Sodin
\cite{sodin2000zeros} (see also \cite[Chapter 13]{Kahane}). Sodin
and Tsirelson \cite{sodin2004random} found the asymptotics of the
variance and proved a central limit theorem for smooth linear statistics
for planar, spherical, and hyperbolic GAFs. More general results about
linear statistics were obtain by Nazarov and Sodin \cite{nazarov2011fluctuations,nazarov2012correlations}.

For the family of hyperbolic GAFs, whose zero sets are invariant with
respect to the isometries of the unit disk, Buckley \cite{buckley2014fluctuations}
found the asymptotics of the variance of the number of zeros (see
also Section \ref{subsec:example_disk}). Buckley and Sodin \cite{buckley2017fluctuations}
studied fluctuations in the increment of the argument along curves
for the planar GAF (GEF). Feldheim \cite{feldheim2018variance} derived
bounds for the growth of the variance of zeros for GAFs which are
invariant with respect to shifts. Ghosh and Peres \cite{GhoshPeres}
showed that the fast decay rate of the variance of smooth linear statistics
of the GEF implies the rigidity of the zero set. Their technique was
recently used by the authors in \cite{kiro2019rigidity} to construct
examples of ``completely rigid'' Gaussian entire functions.

Considerable amount of research is devoted to the study of zero sets
of random algebraic and trigonometric polynomials. Maslova \cite{maslova1974distribution},
Granville and Wigman \cite{granville2011distribution}, Azaïs, Dalmao,
and León \cite{azais2016clt}, and Nguyen and Vu \cite{nguyen2019random}
(by no means an exhaustive list) proved central limit theorems for
real zeros of random polynomials. Bally, Caramellino, and Poly \cite{Non_universality}
studied the dependence of the variance of the number of zeros on the
distribution of the coefficients. Very recently, following the earlier
work \cite{shiffman2010smooth}, Shiffman \cite{Shiffman2020expansion}
found an asymptotic expansion for the variance of smooth statistics
of random zeros on complex manifolds.

\subsection*{Acknowledgments}

We thank Misha Sodin for encouraging us to work on this project, and
for helpful discussions. We thank Aron Wennman for helpful discussions.

\section{Definitions and Preliminaries \label{sec:Results}}

Given an analytic function $G(z)=\sum_{n\geq0}a_{n}^{2}z^{n},$ we
denote its radius of convergence by $R_{G}$, and assume from here
on that $R_{G}\in\left\{ 1,\infty\right\} $. We always assume $a_{n}$
are non-negative, and contain infinitely many non-zero terms (i.e.
$G$ is \emph{transcendental}). We recall the following notation
\[
a\left(z\right)=a_{G}\left(z\right)=z\frac{G^{\prime}\left(z\right)}{G\left(z\right)},\quad b\left(z\right)=b_{G}\left(z\right)=za_{G}^{\prime}\left(z\right).
\]

We use little-$o$ and big-$O$ notation in the standard way. Given
two functions $g_{1},g_{2}:\bbr\to\bbr^{+}$, we write $g_{1}\lesssim g_{2}$
if $g_{1}=O\left(g_{2}\right)$, possibly on a subset of $\bbr$ (depending
on the context). We also write $g_{1}\sim g_{2}$ when $g_{1}\left(x\right)=\left(1+o\left(1\right)\right)g_{2}\left(x\right)$
as $x\to\infty$. Recall that $g_{1}\lesssim_{L}g_{2}$ when there
exists a set $\cN\subset\bbr^{+}$ and a constant $C>0$ so that $g_{1}\left(x\right)\le Cg_{2}\left(x\right)$
for all $x\in\cN$, and $\bbr^{+}\backslash\cN$ is a set of \emph{finite}
logarithmic measure. Let $I\subset\bbr^{+}$ be a open interval, we
denote the fact that $h:I\to\bbr^{+}$ is a \emph{non-decreasing}
and \emph{unbounded} function on $I$ by writing $h\uparrow\infty$.

\subsection{A formula for the variance}

Let $G\left(z\right)=\sum_{n=0}a_{n}^{2}z^{n}$ be the covariance
function of a Gaussian analytic function $f$, with radius of convergence
$R_{G}\in\left\{ 1,\infty\right\} $. For the rest of the paper it
will be convenient to put $e^{t}=r^{2}$, and use the exponential
change of variables
\[
H\left(t\right)=G\left(e^{t}\right)=\sum_{n\ge0}a_{n}^{2}e^{nt},
\]
and also define
\[
t_{G}\eqdef\log R_{G},\quad A\left(t\right)\eqdef\left(\log H\left(t\right)\right)^{\prime}=a\left(e^{t}\right),\quad B\left(t\right)\eqdef A^{\prime}\left(t\right)=b\left(e^{t}\right).
\]
Notice that $A\left(z\right),B\left(z\right)$ are meromorphic functions
which are given by 
\[
A\left(z\right)=\frac{H^{\prime}\left(z\right)}{H\left(z\right)},\qquad B\left(z\right)=\frac{H^{\prime\prime}\left(z\right)}{H\left(z\right)}-\left(\frac{H^{\prime}\left(z\right)}{H\left(z\right)}\right)^{2}=H^{-2}\left(z\right)\left(\sum_{n<m}\left(n-m\right)^{2}a_{n}^{2}a_{m}^{2}e^{\left(m+n\right)z}\right).
\]
We will repeatedly use the following formula for the variance of $n_{f}\left(r\right)$
: 
\begin{align}
\var{n_{f}\left(r\right)} & =\frac{1}{2\pi}\int_{-\pi}^{\pi}\frac{\left|A\left(t+i\theta\right)-A(t)\right|^{2}}{\exp\left(2\int_{0}^{\theta}\text{Im}\left[A(t+i\varphi)\right]\dd\varphi\right)-1}\,\dd\theta,\label{eq:Variance_formula}
\end{align}
for its proof see Claim \ref{clm:var_formula_appendix} in Appendix
\ref{sec:Kahane_formulas} (cf. \cite[p. 195]{Kahane}). We will also
use the following equivalent form
\begin{equation}
\var{n_{f}\left(r\right)}=\frac{1}{2\pi}\int_{-\pi}^{\pi}\frac{\left|H\left(t\right)H^{\prime}\left(t+i\theta\right)-H\left(t+i\theta\right)H^{\prime}\left(t\right)\right|^{2}}{H^{2}\left(t\right)\left(H^{2}\left(t\right)-\left|H^{2}\left(t+i\theta\right)\right|\right)}\,\dd\theta.\label{eq:variance_formula_w_H}
\end{equation}

\subsection{Local admissibility}

In order to bound the integral in (\ref{eq:Variance_formula}) from
below, we will introduce the following definition, which is motivated
by a result from the Wiman-Valiron theory about the value distribution
of entire functions, more precisely the asymptotics of such functions
near their points of maximum modulus (see \cite[Theorem 10]{hayman1974local}).
\begin{defn}
An analytic function $H$ is called \emph{local} \emph{$\delta$-admissible}
on a set $T\subset\left(-\infty,t_{G}\right)$ if there is a function
$\delta\left(t\right):\left[-\infty,t_{G}\right)\to\left(0,\pi\right)$
so that for any $\veps>0$ there exists an $\eta>0$, such that for
$t\in T\cap\left(t_{0}\left(\veps\right),t_{G}\right)$ and $\left|\tau\right|\leq\eta\delta\left(t\right)$
\begin{equation}
\log\frac{H\left(t+\tau\right)}{H\left(t\right)}=\tau A\left(t\right)+\frac{1}{2}\tau^{2}B\left(t\right)+h_{t}\left(\tau\right),\quad\mbox{where }\left|h_{t}\left(\tau\right)\right|\le\veps\left|\tau\right|^{2}B\left(t\right).\label{eq:local_admis_cond}
\end{equation}
\end{defn}

\begin{rem}
It is implicitly assumed that $t+\delta\left(t\right)<t_{G}$ for
all $t<t_{G}$.
\end{rem}

We will show in Section \ref{subsec:proof_loc_admis_non_dec_B} that
if $B\uparrow\infty$, then $H$ is local \emph{$\delta$-}admissible
outside a set of finite logarithmic measure, with $\delta\left(t\right)=\frac{1}{\sqrt{B\left(t\right)}}$.
With a (smaller) appropriate choice of $\delta$ this statement is
also true without making \emph{any} assumptions on $B$, for the details
see Section \ref{subsec:local_admis_no_assump}.

\subsection{\label{subsec:prelim} Lower bound for the variance assuming local
admissibility }

Let $f$ be a Gaussian analytic function with covariance function
$G\left(z\right)=\sum_{n=0}a_{n}^{2}z^{n}$, and radius of convergence
$R_{G}\in\left\{ 1,\infty\right\} $. Recall our notation
\[
H\left(t\right)=G\left(e^{t}\right)=\sum_{n\ge0}a_{n}^{2}e^{nt},
\]
and 
\[
t_{G}=\log R_{G},\quad A\left(t\right)=a\left(e^{t}\right),\quad B\left(t\right)=b\left(e^{t}\right)=A^{\prime}\left(t\right).
\]
In the next lemma we use Cauchy's integral formula to obtain estimates
for $A,B$ when $H$ is a local $\delta$-admissible function.
\begin{lem}
\label{lem:aymp_in_main_sector} Let $H$ be a local $\delta-$admissible
function on $T$ with  $t_{G}\in\left\{ 0,\infty\right\} $. For
any $\veps>0$, there exists $\eta>0$ such that for any $t\in T\cap\left(t_{0}\left(\veps\right),t_{G}\right)$,
and $\left|\theta\right|\le\frac{\eta}{2}\delta\left(t\right)$ we
have
\[
\left(1-\veps\right)B\left(t\right)\left|\theta\right|\le\left|A\left(t+i\theta\right)-A(t)\right|\le\left(1+\veps\right)B\left(t\right)\left|\theta\right|,
\]
and
\[
\frac{1-\veps}{2}\theta^{2}B\left(t\right)\le\int_{0}^{\theta}\Im{A\left(t+i\varphi\right)}d\varphi\le\frac{1+\veps}{2}\theta^{2}B\left(t\right).
\]
\end{lem}

\begin{proof}
Given $\veps>0$, choose $\eta>0$ as in the definition of a local
$\delta-$admissible function, and let $0<\left|\tau\right|\le\frac{\eta}{2}\delta\left(t\right)$.
Differentiating (\ref{eq:local_admis_cond}) with respect to $\tau$
we obtain
\begin{equation}
A\left(t+\tau\right)=\frac{H^{\prime}\left(t+\tau\right)}{H\left(t+\tau\right)}=A\left(t\right)+\tau B\left(t\right)+h_{t}^{\prime}\left(\tau\right),\label{eq:admis_diff}
\end{equation}
with
\[
h_{t}^{\prime}\left(\tau\right)=\frac{1}{2\pi i}\int_{\Gamma}\frac{h_{t}\left(z\right)}{\left(z-\tau\right)^{2}}\dd z,\qquad\text{where }\text{\ensuremath{\Gamma=\left\{ z:\left|z-\tau\right|=\left|\tau\right|\right\} .}}
\]
By local-admissibility we have
\[
\left|h_{t}^{\prime}\left(\tau\right)\right|\le\frac{1}{\left|\tau\right|}\cdot\veps\left|\tau\right|^{2}B\left(t\right)=\veps\left|\tau\right|B\left(t\right).
\]
It follows from (\ref{eq:admis_diff}) that for $t\in T\cap\left(t_{0}\left(\veps\right),t_{G}\right)$
and $\left|\theta\right|\le\frac{\eta}{2}\delta\left(t\right)$ we
have 
\[
\left(1-\veps\right)B\left(t\right)\left|\theta\right|\le\left|A\left(t+i\theta\right)-A(t)\right|\le\left(1+\veps\right)B\left(t\right)\left|\theta\right|.
\]
Since $A\left(t\right)\in\bbr$ we also have there
\begin{align*}
\int_{0}^{\theta}\text{Im}\left[A(t+i\varphi)\right]\dd\varphi & \le\int_{0}^{\theta}\text{Im}\left[A\left(t\right)+i\varphi B\left(t\right)\left(1+\veps\right)\right]\dd\varphi\\
 & =\frac{\left(1+\veps\right)}{2}\theta^{2}B\left(t\right),
\end{align*}
and similarly for the lower bound.
\end{proof}
For $J\subset\bbt$ we define the following integral

\begin{align}
\varint HtJ & \eqdef\frac{1}{2\pi}\int_{J}\frac{\left|A\left(t+i\theta\right)-A(t)\right|^{2}}{\exp\left(2\int_{0}^{\theta}\text{Im}\left[A(t+i\varphi)\right]\dd\varphi\right)-1}\,\dd\theta.\label{eq:Ical_def}
\end{align}

\begin{cor}
\label{cor:lower_bnd_var} Let $H$ be a local $\delta-$admissible
function on $T$ with  $t_{G}\in\left\{ 0,\infty\right\} $, then
for $t\in T$ we have 
\[
\var{n_{f}\left(r\right)}=\varint Ht{\bbt}\gtrsim\min\left\{ \delta\left(t\right)B\left(t\right),\sqrt{B\left(t\right)}\right\} .
\]
\end{cor}

\begin{proof}
Applying Lemma \ref{lem:aymp_in_main_sector} with $\veps=\frac{1}{2}$,
there exists an $\eta>0$ so that for $t$ sufficiently large and
$\left|\theta\right|\leq\frac{\eta}{2}\delta\left(t\right)$,
\[
\frac{\left|A\left(t+i\theta\right)-A(t)\right|^{2}}{\exp\left(2\int_{0}^{\theta}\text{Im}\left[A(t+i\varphi)\right]\dd\varphi\right)-1}\ge\frac{\frac{1}{4}B^{2}\left(t\right)\theta^{2}}{\exp\left(\frac{3}{2}\theta^{2}B\left(t\right)\right)-1}.
\]
Put $\Delta\left(t\right)\eqdef\min\left\{ \frac{\eta}{2}\delta\left(t\right),\frac{1}{\sqrt{B\left(t\right)}}\right\} $,
by the inequality $e^{x}-1\le4x$ which is valid for $x\in\left[0,2\right]$,
we find
\[
\varint Ht{\bbt}\ge\varint Ht{\left[-\Delta\left(t\right),\Delta\left(t\right)\right]}\ge\frac{1}{4}B^{2}\left(t\right)\int_{-\Delta\left(t\right)}^{\Delta\left(t\right)}\frac{\theta^{2}}{6B\left(t\right)\theta^{2}}\,\dd\theta\ge\frac{1}{24}B\left(t\right)\Delta\left(t\right).
\]
\end{proof}

\section{\label{sec:proof_lower_bound} Lower bound for the variance}

In this section we prove Theorem \ref{thm:lower_bound}. First we
assume that $b$ is non-decreasing. Below the letters $\lambda,t,\theta,y$
denote real quantities, and $\tau$ is a complex number. It will be
convenient to put $e^{t}=r^{2}$.

\subsection{\label{subsec:normal_values} Normal values of $t$ and the set $\protect\cX$}

We will now define a set $\cX\subset\bbr^{+}$ whose \emph{complement}
is of \emph{finite} Lebesgue measure, where the function $B$ increases
slowly. Since $B$ is unbounded we may choose a sequence $t_{\ell}\uparrow\infty$
so that 
\[
B\left(t_{\ell}\right)=\ell^{6},\qquad\ell\ge1.
\]
 We then define a sequence of intervals $\left\{ T_{\ell}\right\} _{\ell=1}^{\infty}$
by
\[
T_{\ell}=\left[t_{\ell},t_{\ell+1}\right],\qquad\left|T_{\ell}\right|=t_{\ell+1}-t_{\ell}.
\]

\begin{defn}
The interval $T_{\ell}$ is \emph{long} if
\begin{equation}
\left|T_{\ell}\right|\ge\frac{8}{\ell^{2}},\label{eq:long_interval_cond}
\end{equation}
otherwise it is \emph{short}. For a long interval $T_{\ell}$ we define
its \emph{interior} by
\[
\longint{\ell}=\left[t_{\ell}+\frac{2}{l^{2}},t_{\ell+1}-\frac{2}{\ell^{2}}\right].
\]
\end{defn}

\begin{rem}
Notice that long intervals have a non-trivial interior.
\end{rem}

\begin{defn}
The set $\cX$ of \emph{normal} values of $t$ is given by
\[
\cX\eqdef\bigcup_{T_{\ell}\text{ long}}\longint{\ell}.
\]
\end{defn}

\begin{rem}
Notice that $\sum_{T_{\ell}\text{ short}}\left|T_{\ell}\right|<\infty$
and also $\sum_{T_{\ell}\text{ long}}\left|T_{\ell}\setminus\longint{\ell}\right|<\infty$.
Therefore, the Lebesgue measure of the set $\bbr^{+}\backslash\cX$
is finite. 
\end{rem}

\begin{rem}
Throughout the proof we may need to take the value of $\ell$ to be
sufficiently large, thus we may drop finitely many intervals $\longint{\ell}$
from $\cX$ without explicitly stating it.
\end{rem}

\subsection{Proof of local admissibility for non-decreasing $B$\label{subsec:proof_loc_admis_non_dec_B}}

Let $T_{\ell}$ be a long interval with $\ell\ge4$. Since $B$ is
non-decreasing, we have for all $t\in T_{\ell}$
\begin{equation}
B\left(t_{\ell}\right)\le B\left(t\right)\le B\left(t_{\ell+1}\right)=\left(\ell+1\right)^{6}\le4\ell^{6}=4B\left(t_{\ell}\right).\label{eq:bound_B_in_T_ell}
\end{equation}
By the Lagrange formula for the remainder in the Taylor approximation
for $\log\frac{H\left(y+\lambda\right)}{H\left(y\right)}$ near $\lambda=0$,
we have
\[
\log\frac{H\left(y+\lambda\right)}{H\left(y\right)}=\lambda A\left(y\right)+\frac{\lambda^{2}}{2}B\left(c\right),
\]
where $\left|c-y\right|\le\left|\lambda\right|$. If $y+\lambda,y\in T_{\ell}$,
then $c\in T_{\ell}$, and we deduce that
\begin{equation}
\left|\log\frac{H\left(y+\lambda\right)}{H\left(y\right)}-\lambda A\left(y\right)\right|\le\frac{\lambda^{2}}{2}B\left(c\right)\le2\lambda^{2}B\left(t_{\ell}\right).\label{eq:gen_func_est}
\end{equation}

\subsubsection{An adaptation of Rosenbloom's method}

Recall that
\[
H\left(z\right)=\sum_{n=0}^{\infty}a_{n}^{2}e^{nz},
\]
where $a_{n}$ are non-negative. Following Rosenbloom \cite{rosenbloom1962probability}
we define for $t\in\bbr$, the random variable $X_{t}\in\bbn$ as
follows:
\[
\pr{X_{t}=k}=\frac{a_{k}^{2}e^{kt}}{H\left(t\right)},\quad k\in\bbn.
\]
Then
\[
\ex{X_{t}}=\frac{1}{H\left(t\right)}\sum_{k=0}^{\infty}ka_{k}^{2}e^{kt}=\frac{H^{\prime}\left(t\right)}{H\left(t\right)}=A\left(t\right),
\]
and moreover

\begin{align*}
\var{X_{t}} & =\ex{X_{t}^{2}}-\left(\ex{X_{t}}\right)^{2}=\frac{1}{H\left(t\right)}\sum_{k=0}^{\infty}k^{2}a_{k}^{2}e^{kt}-A^{2}\left(t\right)\\
 & =\frac{H^{\prime\prime}\left(t\right)}{H\left(t\right)}-\left(\frac{H^{\prime}\left(t\right)}{H\left(t\right)}\right)^{2}=B\left(t\right).
\end{align*}
In order to approximate $H$ by an appropriate (exponential) polynomial,
we first prove the following lemma.
\begin{lem}
\label{lem:poly_approx}For $t\in\longint{\ell}$, and $\left|\tau-t\right|<\frac{1}{\sqrt{B\left(t\right)}}$,
we have
\[
\left|E\left(\tau\right)\right|\eqdef\left|\sum_{\left|k-A\left(t\right)\right|>s\sqrt{B\left(t\right)}}a_{k}^{2}e^{k\text{\ensuremath{\tau}}}\right|\le2H\left(\Re{\tau}\right)\exp\left(-\frac{1}{8}\left(s-4\right)^{2}\right),
\]
for all $4<s<B^{1/6}\left(t\right)$.
\end{lem}

\begin{proof}
Since $\left|a_{k}^{2}e^{k\tau}\right|=a_{k}^{2}e^{k\Re{\tau}}$,
by the triangle inequality we may assume $\tau$ is real. Notice that
\[
\sum_{\left|k-A\left(t\right)\right|>s\sqrt{B\left(t\right)}}a_{k}^{2}e^{k\tau}=H\left(\tau\right)\pr{\left|X_{\tau}-A\left(t\right)\right|>s\sqrt{B\left(t\right)}},
\]
and also
\[
\ex{e^{\lambda X_{\tau}}}=\frac{1}{H\left(\tau\right)}\sum_{k=0}^{\infty}e^{\lambda k}a_{k}^{2}e^{k\tau}=\frac{H\left(\tau+\lambda\right)}{H\left(\tau\right)}.
\]
Fix $\lambda>0$ so that $\tau+\lambda\in T_{\ell}$. By Markov's
inequality
\begin{align*}
\pr{X_{\tau}>A\left(t\right)+s\sqrt{B\left(t\right)}} & \le\frac{\ex{e^{\lambda X_{\tau}}}}{\exp\left(\lambda\left(A\left(t\right)+s\sqrt{B\left(t\right)}\right)\right)}\\
 & =\exp\left(\log\frac{H\left(\tau+\lambda\right)}{H\left(\tau\right)}-\lambda\left(A\left(t\right)+s\sqrt{B\left(t\right)}\right)\right).
\end{align*}
By (\ref{eq:gen_func_est}) and (\ref{eq:bound_B_in_T_ell}) we have
\begin{align*}
\log\frac{H\left(\tau+\lambda\right)}{H\left(\tau\right)}-\lambda A\left(t\right) & =\log\frac{H\left(\tau+\lambda\right)}{H\left(\tau\right)}-\lambda A\left(\tau\right)+\lambda\left[A\left(\tau\right)-A\left(t\right)\right]\\
 & \le\lambda\left[A\left(\tau\right)-A\left(t\right)\right]+2\lambda^{2}B\left(t_{\ell}\right)\\
 & \le4\lambda\left|\tau-t\right|B\left(t_{\ell}\right)+2\lambda^{2}B\left(t_{\ell}\right)\le4\lambda\sqrt{B\left(t\right)}+2\lambda^{2}B\left(t\right).
\end{align*}
Therefore, by taking $\lambda=\frac{s-4}{4\sqrt{B\left(t\right)}}$,
we get
\begin{align*}
\log\pr{X_{\tau}>A\left(t\right)+s\sqrt{B\left(t\right)}} & \le4\lambda\sqrt{B\left(t\right)}+2\lambda^{2}B\left(t\right)-\lambda s\sqrt{B\left(t\right)}\\
 & \le-\frac{1}{8}\left(s-4\right)^{2}.
\end{align*}
Since $s<B^{1/6}\left(t\right)$ we obtain by (\ref{eq:bound_B_in_T_ell})
\begin{align*}
\tau+\lambda & <t_{\ell+1}-\frac{2}{\ell^{2}}+\frac{1}{\sqrt{B\left(t\right)}}+\frac{1}{4B\left(t\right)^{1/3}}\le t_{\ell+1}-\frac{2}{\ell^{2}}+\frac{1}{\sqrt{B\left(t_{\ell}\right)}}+\frac{1}{4B\left(t_{\ell}\right)^{1/3}}\\
 & =t_{\ell+1}-\frac{2}{\ell^{2}}+\frac{1}{\ell^{3}}+\frac{1}{4\ell^{2}}<t_{\ell+1},
\end{align*}
in the same way we see that $\tau+\lambda>t_{\ell}$, so that $\tau+\lambda\in T_{\ell}$
as required. The bound for $\pr{X_{t}<A\left(t\right)+s\sqrt{B\left(t\right)}}$
is obtained similarly.
\end{proof}
\begin{defn}
We say that an exponential polynomial with \emph{real} exponents is
of \emph{width} at most $\omega$, if it is of the form
\[
\sum_{k=0}^{m}c_{k}e^{\alpha_{k}z},\quad\text{with }\max_{k}\left|\alpha_{k}\right|\le\omega.
\]
\end{defn}

The following lemma adapts \cite[Lemma 8]{hayman1974local} to exponential
polynomials.
\begin{lem}
\label{lem:exp_poly_near_max} Let $P\left(z\right)$ be an exponential
polynomial of width at most $\omega$, and non-negative coefficients.
We have for any $t\in\bbr$, and $\left|\tau-t\right|\le\frac{1}{5\omega}$
that
\[
\frac{3}{4}P\left(t\right)\le\left|P\left(\tau\right)\right|\le\frac{5}{4}P\left(t\right).
\]
\end{lem}

\begin{proof}
Suppose $P\left(z\right)=\sum_{k=0}^{m}c_{k}e^{\alpha_{k}z}$ with
$\max_{k}\left|\alpha_{k}\right|\le\omega$, and $c_{k}\ge0$. Then
\begin{align*}
\left|P^{\prime}\left(\tau\right)\right| & =\left|\sum_{k=0}^{m}c_{k}\alpha_{k}e^{\alpha_{k}\tau}\right|=\left|\sum_{k=0}^{m}c_{k}e^{\alpha_{k}t}\alpha_{k}e^{\alpha_{k}\left(\tau-t\right)}\right|\le\max_{k}\left|\alpha_{k}e^{\alpha_{k}\left(\Re{\tau}-t\right)}\right|\cdot P\left(t\right)\\
 & \le\max_{k}\left|\alpha_{k}\right|e^{\left|\alpha_{k}\right|\left|\tau-t\right|}\cdot P\left(t\right)\le\omega e^{\omega\left|\tau-t\right|}P\left(t\right).
\end{align*}
Now,
\[
\left|P\left(\tau\right)-P\left(t\right)\right|=\left|\int_{t}^{\tau}P^{\prime}\left(s\right)\,\dd s\right|\le\omega e^{\omega\left|\tau-t\right|}P\left(t\right)\cdot\left|\tau-t\right|\le\frac{1}{5}e^{1/5}P\left(t\right)\le\frac{1}{4}P\left(t\right).
\]
\end{proof}

\subsubsection{Proof of second part of Theorem \ref{thm:lower_bound}}

The proof of Theorem \ref{thm:lower_bound} in the case $b\uparrow\infty$
follows by combining the next proposition with Corollary \ref{cor:lower_bnd_var}.
We will use the Schwarz integral formula (see \cite[Chapter 4, Section 6.3]{ahlforsComplex})\emph{,
if $g$ is an analytic function in the closed unit disk, then
\begin{equation}
g\left(z\right)=\frac{1}{2\pi i}\int_{\bbt}\frac{\zeta+z}{\zeta-z}\text{\ensuremath{\Re{g\left(\zeta\right)}\,}\ensuremath{\ensuremath{\frac{\dd\ensuremath{\zeta}}{\zeta}}+i\ensuremath{\Im{g\left(0\right)}}}},\label{eq:Schwarz_int_form}
\end{equation}
for all $\left|z\right|<1$, where $\bbt$ is the unit circle.}
\begin{prop}[{cf. \cite[Theorem 10]{hayman1974local}}]
\label{prop:rosenbloom_local_adm} If $B\left(t\right)\uparrow\infty$,
then $H$ is\emph{ local} \emph{$\delta$-admissible} on $\cX$, with
$\delta\left(t\right)=\frac{1}{\sqrt{B\left(t\right)}}$, that is,
for any $\veps>0$ there exists $\eta>0$, such that for $t\in\cX\cap\left(t_{0}\left(\veps\right),t_{G}\right)$
and $\left|w\right|\leq\frac{\eta}{\sqrt{B\left(t\right)}}$ we have
\[
\log\frac{H\left(t+w\right)}{H\left(t\right)}=wA\left(t\right)+\frac{1}{2}w^{2}B\left(t\right)+h_{t}\left(w\right),\quad\mbox{where }\left|h_{t}\left(w\right)\right|\le\veps\left|w\right|^{2}B\left(t\right).
\]
\end{prop}

\begin{proof}
Let $\veps>0$ be given, and $\eta>0$ depending on $\veps$ to be
chosen later, also let $t\in\cX$. Put
\[
Q\left(\tau\right)\eqdef\sum_{\left|k-A\left(t\right)\right|\le s\sqrt{B\left(t\right)}}a_{k}^{2}e^{k\text{\ensuremath{\tau}}}=H\left(\tau\right)-E\left(\tau\right),
\]
then by Lemma \ref{lem:poly_approx} for $\left|\tau-t\right|<\frac{1}{\sqrt{B\left(t\right)}}$
we have $\left|E\left(\tau\right)\right|\le2H\left(\Re{\tau}\right)\exp\left(-\frac{1}{8}\left(s-4\right)^{2}\right)$
for all $4<\left|s\right|<B^{1/6}\left(t\right)$. We also write $P\left(\tau\right)=e^{A\left(t\right)\tau}Q\left(\tau\right)$
so that $P$ is an exponential polynomial with non-negative coefficients
of width at most $s\sqrt{B\left(t\right)}$. Choosing $s=9$ so that
$2\exp\left(-\frac{1}{8}\left(s-4\right)^{2}\right)<\frac{1}{4}$,
we have
\begin{equation}
\left|E\left(\tau\right)\right|<\frac{1}{4}H\left(\Re{\tau}\right),\label{eq:E_tau_bound}
\end{equation}
and in particular
\begin{equation}
\frac{3}{4}H\left(\Re{\tau}\right)<Q\left(\Re{\tau}\right)=H\left(\Re{\tau}\right)-E\left(\Re{\tau}\right)<\frac{5}{4}H\left(\Re{\tau}\right).\label{eq:Q_re_tau_bounds}
\end{equation}
By Lemma \ref{lem:exp_poly_near_max} and the previous inequality
for $\Re{\tau}=t$, we have for $\left|\tau-t\right|\le\frac{1}{45\sqrt{B\left(t\right)}}$
\begin{equation}
\left(\frac{3}{4}\right)^{2}H\left(t\right)\le\frac{3}{4}Q\left(t\right)\le\left|Q\left(\tau\right)\right|\le\frac{5}{4}Q\left(t\right)\le\frac{5}{4}H\left(t\right).\label{eq:Q_tau_bounds}
\end{equation}
From (\ref{eq:E_tau_bound}), (\ref{eq:Q_re_tau_bounds}), and (\ref{eq:Q_tau_bounds}),
we get
\[
\left|E\left(\tau\right)\right|\le\frac{1}{4}\cdot\frac{4}{3}Q\left(\Re{\tau}\right)\le\frac{1}{3}\cdot\frac{5}{4}H\left(t\right).
\]
Thus, we have

\[
\left|H\left(\tau\right)\right|\le\left|Q\left(\tau\right)\right|+\left|E\left(\tau\right)\right|\le\left[\frac{5}{4}+\frac{1}{3}\cdot\frac{5}{4}\right]H\left(t\right)\le2H\left(t\right),
\]
and
\[
\left|H\left(\tau\right)\right|\ge\left|Q\left(\tau\right)\right|-\left|E\left(\tau\right)\right|\ge\left[\left(\frac{3}{4}\right)^{2}-\frac{1}{3}\cdot\frac{5}{4}\right]H\left(t\right)\ge\frac{1}{8}H\left(t\right).
\]
We found that
\begin{equation}
-\log8\le\log\frac{\left|H\left(\tau\right)\right|}{H\left(t\right)}\le\log2.\label{eq:real_part_bound}
\end{equation}

Now let us define the analytic function
\[
\phi\left(w\right)=\log H\left(t+w\right)-\log H\left(t\right),\quad\phi\left(w\right)=\sum_{n=1}^{\infty}\phi_{n}w^{n},\qquad\left|w\right|\le\frac{1}{45\sqrt{B\left(t\right)}}\defeq\lambda_{0}.
\]
By (\ref{eq:real_part_bound}) we have that $\left|\Re{\phi\left(w\right)}\right|\le\log8$,
and therefore by (\ref{eq:Schwarz_int_form}) for $\left|w\right|\le\frac{1}{2}\lambda_{0}$
we get (since $\phi\left(0\right)=0$)
\[
\left|\phi\left(w\right)\right|\le\frac{1+\frac{1}{2}}{1-\frac{1}{2}}\cdot\log8<7.
\]
By Cauchy's estimates $\left|\phi_{n}\right|\le7\left(\frac{2}{\lambda_{0}}\right)^{n}$,
and therefore, for $\left|w\right|\le\frac{1}{4}\lambda_{0}$
\begin{align*}
\left|\phi\left(w\right)-wA\left(t\right)-\frac{1}{2}w^{2}B\left(t\right)\right| & =\left|\sum_{n=3}^{\infty}\phi_{n}w^{n}\right|\\
 & \le7\sum_{n=3}^{\infty}\left(\frac{2w}{\lambda_{0}}\right)^{n}\le7\cdot8\left(\frac{w}{\lambda_{0}}\right)^{3}\sum_{n=0}^{\infty}\frac{1}{2^{n}}=112\left(\frac{w}{\lambda_{0}}\right)^{3}.
\end{align*}
Thus in order to obtain the result it remains to choose $\eta=\frac{\veps}{112\cdot45^{3}}$.
\end{proof}

\subsubsection{\label{subsec:local_admis_no_assump} Proof of first part of Theorem
\ref{thm:lower_bound}}

Fix an entire function $G(z)=\sum_{n\geq0}a_{n}^{2}z^{n}$, with non-negative
coefficients $a_{n}$, and recall that
\[
a\left(r\right)=r\frac{G^{\prime}\left(r\right)}{G\left(r\right)},\quad b\left(r\right)=ra^{\prime}\left(r\right).
\]
In order to get a lower bound for the variance without any assumptions
on the function $b\left(r\right)$ we will use some results about
$G$ obtained by the Wiman-Valiron method (see \cite{hayman1974local}).
We recall some of the terminology regarding entire functions: $\mu\left(r\right)=\max_{n}\left\{ a_{n}^{2}r^{n}\right\} $
is called the \emph{maximal term} of $G$, and $N\left(r\right)=\max\left\{ n:a_{n}^{2}r^{n}=\mu\left(r\right)\right\} $
the \emph{central index}.

One of the main results of the Wiman-Valiron method is that there
is a set $\cN\subset\bbr^{+}$, such that $G\left(z\right)$ has desirable
properties if $\left|z\right|\in\cN$, and that $\bbr^{+}\backslash\cN$
has finite logarithmic measure (see \cite[Sections 2 and 3]{hayman1974local}).
We fix a parameter $\gamma\in\left(0,\tfrac{1}{2}\right)$; the set
$\cN$ will depend on $G$ and $\gamma$. By \cite[Theorem 2]{hayman1974local}
for $r\in\cN$ we have for all $n\in\bbn$
\[
a_{n}^{2}r^{n}\le\mu\left(r\right)\exp\left(-\frac{k^{2}}{\left(\left|k\right|+N\left(r\right)\right)^{1+\gamma}}\right),\quad\mbox{where }k=n-N\left(r\right),
\]
hence the summands $a_{n}^{2}\left|z\right|^{n}$ of the series $G\left(z\right)$
with $\left|z\right|=r$, corresponding to the ``window'' of indices

\begin{equation}
\left\{ n\in\bbn:\left|n-N\left(r\right)\right|\le K\left(r\right)\right\} \quad\text{with }K\left(r\right)\eqdef N\left(r\right)^{\frac{1+\gamma}{2}},\label{eq:K(r)_def}
\end{equation}
are the largest ones. In particular, this implies that $N\left(r\right)$
and $a\left(r\right)$ are asymptotically comparable, see Claim \ref{clm:approx_N_by_a}
below. Notice that by applying the change of variables $x\mapsto\sqrt{x}$,
the set $\left\{ r>0:r^{2}\notin\cN\right\} $ is of finite logarithmic
measure as well.

It is known that for any $\gamma>0$ we have
\begin{equation}
b\left(r\right)\le a\left(r\right)^{1+\gamma}\label{eq:b_bound_in_a}
\end{equation}
outside a set of finite logarithmic measure (see \cite[Lemma 1]{rosenbloom1962probability}).
Thus, we may and will assume that (\ref{eq:b_bound_in_a}) is satisfied
for all $r\in\cN$.
\begin{claim}
\label{clm:approx_N_by_a} For $r\in\cN$ we have
\[
a\left(r\right)=N\left(r\right)+O\left(K\left(r\right)\right),\quad\mbox{as }r\to\infty.
\]
\end{claim}

\begin{proof}
Since $a_{n}^{2}\ge0$ we have that $\max_{\left|z\right|=r}\left|G\left(z\right)\right|=G\left(r\right)$
and hence, by \cite[Theorem 12]{hayman1974local} with $q=1,f=G,k=K\left(r\right)$
we have
\[
\frac{r}{N\left(r\right)}G^{\prime}\left(r\right)=G\left(r\right)+O\left(\frac{K\left(r\right)}{N\left(r\right)}\right)G\left(r\right).
\]
and therefore
\[
a\left(r\right)=N\left(r\right)+O\left(K\left(r\right)\right).
\]
\end{proof}
Fixing $r\in\cN$ we can approximate the function $G$ near $\left|z\right|=r$
by a polynomial of degree about $K\left(r\right)$ (cf. Lemma \ref{lem:poly_approx}).
This allows us to obtain rather precise Taylor expansion asymptotics
for $\log G\left(e^{\tau}\right)$ near $\tau=2\log r$.

The following lemma is a special case of \cite[Lemma 2]{hayman1974local}.
\begin{lem}
Suppose $e^{t}=r^{2}\in\cN$ and $K,N$ are as above, then for $\left|\tau-t\right|\le\frac{2}{K\left(e^{t}\right)}$
\[
\left|\sum_{\left|k-N\left(e^{t}\right)\right|>K\left(e^{t}\right)}a_{k}^{2}e^{k\text{\ensuremath{\tau}}}\right|\le\frac{1}{4}H\left(\Re{\tau}\right).
\]
\end{lem}

Combining the lemma above with Claim \ref{clm:approx_N_by_a}, we
find that we can replace $N$ by $A$.
\begin{cor}
There exists a constant $s>0$, so that for $e^{t}=r^{2}\in\cN$ and
$\left|\tau-t\right|\le\frac{2}{K\left(e^{t}\right)}$,
\[
\left|\sum_{\left|k-A\left(t\right)\right|>sK\left(e^{t}\right)}a_{k}^{2}e^{k\text{\ensuremath{\tau}}}\right|\le\frac{1}{4}H\left(\Re{\tau}\right).
\]
\end{cor}

Repeating the proof of Proposition \ref{prop:rosenbloom_local_adm}
using the previous corollary instead of Lemma \ref{lem:poly_approx}
we obtain the following result.
\begin{prop}
Any entire function $G$ with non-negative coefficients is local $\delta-$admissible
on a set $\cN=\cN_{G}$ with $\delta\left(t\right)=\frac{B\left(t\right)}{K^{3}\left(e^{t}\right)}$.
Here $\bbr^{+}\backslash\cN$ is a set of finite logarithmic measure.
\end{prop}

We are now ready to prove the first part of Theorem \ref{thm:lower_bound}.
We recall that $f$ is a Gaussian entire function with covariance
kernel $G$.
\begin{proof}[Proof of Theorem \ref{thm:lower_bound}]
Choose $\gamma$ sufficiently small so that $\frac{3}{2}\left(1+\gamma\right)+\gamma\le\frac{3}{2}+\veps$.
By (\ref{eq:b_bound_in_a}), (\ref{eq:K(r)_def}), and Claim \ref{clm:approx_N_by_a},
we have
\[
\frac{B^{2}\left(t\right)}{K^{3}\left(e^{t}\right)}\le\frac{2B^{2}\left(t\right)}{A\left(t\right)^{\frac{3}{2}\left(1+\gamma\right)}}\le2\sqrt{B\left(t\right)}\quad\text{as }t\to\infty,\quad e^{t}\in\cN,
\]
and therefore by Corollary \ref{cor:lower_bnd_var}, with $e^{t}=r^{2}$,
$H\left(z\right)=G\left(e^{z}\right)$, and the previous proposition
\[
\var{n_{f}\left(r\right)}=\varint Ht{\bbt}\ge\min\left\{ \delta\left(t\right)B\left(t\right),\sqrt{B\left(t\right)}\right\} \gtrsim\frac{B^{2}\left(t\right)}{K^{3+\gamma}\left(e^{t}\right)}\gtrsim\frac{b^{2}\left(r^{2}\right)}{a^{\frac{3}{2}\left(1+\gamma\right)+\gamma}\left(r^{2}\right)}.
\]
 
\end{proof}

\section{Asymptotics of the variance - proof of Theorem~\ref{thm:var_asymp}\label{sec:variance_asymp}}

Let $f$ be a Gaussian analytic function with covariance function
$G$, recall $H\left(t\right)=G\left(e^{t}\right)$, and that
\[
t_{G}=\log R_{G}\in\left\{ 0,\infty\right\} ,\quad A\left(t\right)=\left(\log H\left(t\right)\right)^{\prime}=a\left(e^{t}\right),\quad B\left(t\right)=A^{\prime}\left(t\right)=b\left(e^{t}\right).
\]
In this section we find the asymptotic growth of $\var{n_{f}\left(r\right)}$
as $r\to R_{G}^{-}$ when the covariance kernel $G$ is type I admissible.

\subsection{Type I admissible covariance functions\label{subsec: type I admissible}}

To find precise asymptotics for the integral (\ref{eq:Variance_formula})
we make certain assumptions on the function $H$, motivated by the
Hayman \cite{Hayman} admissibility condition.
\begin{defn}
We call $G$ \emph{type I admissible} if the function $H$ has the
following properties: 
\end{defn}

\begin{enumerate}
\item \label{enu:Hayman_Admis_Assum_growth_B} $B\left(t\right)\to\infty$
as $t\to t_{G}^{-}$.
\item \label{enu:Assump_A_control_by_B} $A\left(t\right)=O\left(B^{2}\left(t\right)\right)$
as $t\to t_{G}^{-}$.
\item \label{enu:Hayman_Admis_Assum_Asymp} There is a constant $C_{G}>2$
such that 
\[
\log\frac{H\left(t+i\theta\right)}{H\left(t\right)}=i\theta A\left(t\right)-\frac{1}{2}\theta^{2}B\left(t\right)\left(1+o\left(1\right)\right),\quad\text{as }t\to t_{G}^{-},
\]
uniformly for all $\left|\theta\right|\le\delta(t):=\sqrt{C_{G}\frac{\log B\left(t\right)}{B\left(t\right)}}$.
\item \label{enu:Hayman_Admis_Assum_decay} $\left|H\left(t+i\theta\right)\right|=O\left(\frac{H\left(t\right)}{B^{2}\left(t\right)}\right)$
and $\left|H^{\prime}\left(t+i\theta\right)\right|=O\left(\frac{H^{\prime}\left(t\right)}{B^{2}\left(t\right)}\right)$
as $t\to t_{G}^{-}$, uniformly in $\left|\theta\right|\in\left[\delta\left(t\right),\pi\right]$.
\end{enumerate}

\begin{rem}
By the proof of \cite[Lemma 4]{Hayman} it follows that an admissible
function is local $\delta$-admissible on $\bbr^{+}$ with $\delta\left(t\right)=\frac{c}{\sqrt{B\left(t\right)}}$
with some constant $c>0$.
\end{rem}

\begin{rem}
Since $B=A^{\prime}=\left(\log H\right)^{\prime\prime}$ it follows
that Assumption \ref{enu:Assump_A_control_by_B} puts a restriction
on the minimal growth rate of the function $H$.
\end{rem}

\begin{rem}
The choice of the constant $C_{G}$ is not important, we choose it
in this way so that Assumptions \ref{enu:Hayman_Admis_Assum_Asymp}
and \ref{enu:Hayman_Admis_Assum_decay} will agree for $\theta=\delta\left(t\right)$.
\end{rem}

\subsection{Asymptotics of the variance}

We put $r^{2}=e^{t}$, and split the integral in (\ref{eq:Variance_formula})
into two parts:
\begin{align*}
\text{Var}\left(n_{f}\left(r\right)\right) & =\varint Ht{\bbt}=\frac{1}{2\pi}\int_{-\pi}^{\pi}\frac{\left|A\left(t+i\theta\right)-A\left(t\right)\right|^{2}}{\exp\left(2\int_{0}^{\theta}\text{Im}\left[A\left(t+i\varphi\right)\right]\dd\varphi\right)-1}\,\dd\theta\\
 & =\varint Ht{\left[-\delta\left(t\right),\delta\left(t\right)\right]}+\varint Ht{\bbt\setminus\left[-\delta\left(t\right),\delta\left(t\right)\right]}\\
 & \defeq J_{1}\left(r\right)+J_{2}\left(r\right),
\end{align*}
where we used the definition of $\cI$ in (\ref{eq:Ical_def}).

\subsubsection{Evaluating $J_{1}$}

For $\left|\theta\right|\le\delta\left(t\right)$, Assumption \ref{enu:Hayman_Admis_Assum_Asymp}
implies
\begin{equation}
\left|A\left(t+i\theta\right)-A(t)\right|=B\left(t\right)\left|\theta\right|\left(1+o\left(1\right)\right),\quad t\to t_{G}^{-},\label{eq: g asym 1}
\end{equation}
uniformly in $\left|\theta\right|\leq\delta\left(t\right)$. Since
$A\left(t\right)\in\bbr$
\begin{align}
\int_{0}^{\theta}\text{Im}\left[A(t+i\varphi)\right]\dd\varphi & =\int_{0}^{\theta}\text{Im}\left[A\left(t\right)+i\varphi B\left(t\right)\left(1+o\left(1\right)\right)\right]\dd\varphi\nonumber \\
 & \sim-\frac{1}{2}\theta^{2}B\left(t\right),\quad t\to t_{G}^{-},\label{eq: g asym 2}
\end{align}
also uniformly in $\left|\theta\right|\leq\delta\left(t\right)$ .
Making the change of variables
\[
u=\text{\ensuremath{\sqrt{-2\int_{0}^{\theta}\text{Im}\left[A(t+i\varphi)\right]\dd\varphi}\sim\theta\sqrt{B\left(t\right)}}}
\]
yields by (\ref{eq: g asym 1}) and (\ref{eq: g asym 2})
\[
J_{1}\left(r\right)\sim\frac{\sqrt{B\left(t\right)}}{2\pi}\int_{\left|u\right|\le\delta\left(t\right)\sqrt{B\left(t\right)\left(1+o\left(1\right)\right)}}\frac{u^{2}}{\exp\left(u^{2}\right)-1}\,\dd u.
\]
By Assumption \ref{enu:Hayman_Admis_Assum_growth_B} and the definition
of $\delta$, we have that $\delta\left(t\right)\sqrt{B\left(t\right)}\to\infty,$
as $t\to t_{G}^{-}.$ Thus 
\begin{align*}
\int_{\left|u\right|\leq\delta\left(t\right)\sqrt{B\left(t\right)\left(1+o\left(1\right)\right)}}\frac{u^{2}}{\exp\left(u^{2}\right)-1}\,\dd u & \sim\int_{\bbr}\frac{u^{2}}{\exp\left(u^{2}\right)-1}\,\dd u=\frac{\sqrt{\pi}}{2}\zeta\left(\tfrac{3}{2}\right),\qquad\text{as }t\to t_{G}^{-},
\end{align*}
where $\zeta$ is the Riemann zeta function. This gives the main term
in the asymptotic behavior of the variance
\begin{equation}
J_{1}\left(r\right)\sim\frac{\zeta\left(\frac{3}{2}\right)}{4\sqrt{\pi}}\sqrt{B\left(t\right)},\quad t\to t_{G}^{-}.\label{eq:J_1_asymp}
\end{equation}

\subsubsection{Bounding $J_{2}$}

Again we write $r^{2}=e^{t}$. The admissibility assumptions allow
us to control the size of $G$ also in the range $\delta\left(t\right)\leq\left|\theta\right|\leq\pi$.
Note the identity
\begin{align*}
\frac{\left|A\left(t+i\theta\right)-A(t)\right|^{2}}{\exp\left(2\int_{0}^{\theta}\text{Im}\left[A(t+i\varphi)\right]\dd\varphi\right)-1} & =\frac{\left|H\left(t\right)H^{\prime}\left(t+i\theta\right)-H\left(t+i\theta\right)H^{\prime}\left(t\right)\right|^{2}}{H^{2}\left(t\right)\left(H^{2}\left(t\right)-\left|H^{2}\left(t+i\theta\right)\right|\right)}.
\end{align*}
By Assumptions \ref{enu:Hayman_Admis_Assum_growth_B} and \ref{enu:Hayman_Admis_Assum_decay},
we have that $H^{2}\left(t\right)-\left|H^{2}\left(t+i\theta\right)\right|\sim H^{2}\left(t\right)$.
Therefore, again by Assumption \ref{enu:Hayman_Admis_Assum_decay}
we get 
\begin{align*}
\frac{\left|A\left(t+i\theta\right)-A(t)\right|^{2}}{\exp\left(2\int_{0}^{\theta}\text{Im}\left[A(t+i\varphi)\right]\dd\varphi\right)-1} & \le\left(2+o\left(1\right)\right)\left[\frac{\left|H\left(t\right)H^{\prime}\left(t+i\theta\right)\right|^{2}}{H^{4}\left(t\right)}+\frac{\left|H\left(t+i\theta\right)H^{\prime}\left(t\right)\right|^{2}}{H^{4}\left(t\right)}\right]\\
 & =\left(2+o\left(1\right)\right)\left[\frac{\left|H^{\prime}\left(t+i\theta\right)\right|^{2}}{H^{2}\left(t\right)}+\frac{\left|H\left(t+i\theta\right)\right|^{2}}{H^{2}\left(t\right)}A^{2}\left(t\right)\right]\\
 & =O\left(\frac{A^{2}\left(t\right)}{B^{4}\left(t\right)}\right).
\end{align*}
Finally, by Assumption \ref{enu:Assump_A_control_by_B}

\[
J_{2}\left(r\right)=O\left(\frac{A^{2}\left(t\right)}{B^{4}\left(t\right)}\right)=O\left(1\right),\quad t\to t_{G}^{-},
\]
and combining this with (\ref{eq:J_1_asymp}) we conclude that
\[
\var{n_{f}\left(r\right)}=J_{1}\left(r\right)+J_{2}\left(r\right)\sim\frac{\zeta\left(\frac{3}{2}\right)}{4\sqrt{\pi}}\sqrt{b\left(r^{2}\right)},\quad r\to R_{G}^{-},
\]
thus proving Theorem \ref{thm:var_asymp}.\qed

\section{Upper bound for the variance - proof of Theorem \ref{thm:upper_bound}}

The upper bound for the variance is derived from the following algebraic
identity.
\begin{claim}
\label{clm:identity} For $a_{1},a_{2},a_{3}\in\bbr$ and $b_{1},b_{2},b_{3}\in\bbc$
the following holds
\[
\left|a_{1}b_{3}-\overline{b_{1}}b_{2}\right|^{2}=\left(a_{1}a_{3}-\left|b_{2}\right|^{2}\right)\left(a_{1}a_{2}-\left|b_{1}\right|^{2}\right)-a_{1}\cdot\det\left(\begin{array}{ccc}
a_{1} & \overline{b_{1}} & \overline{b_{2}}\\
b_{1} & a_{2} & \overline{b_{3}}\\
b_{2} & b_{3} & a_{3}
\end{array}\right).
\]
\end{claim}

\begin{proof}
One can directly show this equality holds, but here we will give a
geometric reasoning which holds when the matrix in the expression
above is positive definite (which is the case we use).In that case
we may take vectors $\left(v_{1},v_{2},v_{3}\right)$ so that their
Gram matrix satisfies
\begin{equation}
\left(\begin{array}{ccc}
a_{1} & \overline{b_{1}} & \overline{b_{2}}\\
b_{1} & a_{2} & \overline{b_{3}}\\
b_{2} & b_{3} & a_{3}
\end{array}\right)=\left(\begin{array}{ccc}
\left\langle v_{1},v_{1}\right\rangle  & \left\langle v_{2},v_{1}\right\rangle  & \left\langle v_{3},v_{1}\right\rangle \\
\left\langle v_{1},v_{2}\right\rangle  & \left\langle v_{2},v_{2}\right\rangle  & \left\langle v_{3},v_{2}\right\rangle \\
\left\langle v_{1},v_{3}\right\rangle  & \left\langle v_{2},v_{3}\right\rangle  & \left\langle v_{3},v_{3}\right\rangle 
\end{array}\right)\label{eq:Gram_matrix_rep}
\end{equation}
and we denote the corresponding Gram determinant by $\mathrm{Gram}\left(v_{1},v_{2},v_{3}\right)$.
On one side we have
\[
\mathrm{dist}^{2}\left(v_{3},\mathrm{span}\left\{ v_{1},v_{2}\right\} \right)=\frac{\mathrm{Gram}\left(v_{1},v_{2},v_{3}\right)}{\mathrm{Gram}\left(v_{1},v_{2}\right)}=\frac{\mathrm{Gram}\left(v_{1},v_{2},v_{3}\right)}{a_{1}a_{2}-\left|b_{1}\right|^{2}}.
\]
On the other hand, using the Gram-Schmidt process to find orthogonal
vectors $\left(w_{1},w_{2},w_{3}\right)$ we get
\begin{align*}
w_{1} & =v_{1},\\
w_{2} & =v_{2}-\frac{\left\langle v_{2},w_{1}\right\rangle }{\left\langle w_{1},w_{1}\right\rangle }w_{1},\\
w_{3} & =v_{3}-\frac{\left\langle v_{3},w_{1}\right\rangle }{\left\langle w_{1},w_{1}\right\rangle }w_{1}-\frac{\left\langle v_{3},w_{2}\right\rangle }{\left\langle w_{2},w_{2}\right\rangle }w_{2}.
\end{align*}
Taking into account (\ref{eq:Gram_matrix_rep}), we find that
\[
w_{1}=v_{1},\,w_{2}=v_{2}-\frac{\overline{b_{1}}}{a_{1}}v_{1},\,\left\langle w_{2},w_{2}\right\rangle =a_{2}-\frac{\left|b_{1}\right|^{2}}{a_{1}}=\frac{a_{1}a_{2}-\left|b_{1}\right|^{2}}{a_{1}}
\]
and
\[
\left\langle v_{3},w_{2}\right\rangle =\left\langle v_{3},v_{2}-\frac{\overline{b_{1}}}{a_{1}}v_{1}\right\rangle =\left\langle v_{3},v_{2}\right\rangle -\frac{b_{1}}{a_{1}}\left\langle v_{3},v_{1}\right\rangle =\overline{b_{3}}-\frac{b_{1}\overline{b_{2}}}{a_{1}}=\frac{\overline{a_{1}b_{3}-\overline{b_{1}}b_{2}}}{a_{1}}.
\]
Finally

\begin{align*}
\mathrm{dist}^{2}\left(v_{3},\mathrm{span}\left\{ v_{1},v_{2}\right\} \right) & =\left\langle v_{3},v_{3}\right\rangle -\frac{\left|\left\langle v_{3},w_{1}\right\rangle \right|^{2}}{\left\langle w_{1},w_{1}\right\rangle }-\frac{\left|\left\langle v_{3},w_{2}\right\rangle \right|^{2}}{\left\langle w_{2},w_{2}\right\rangle }\\
 & =a_{3}-\frac{\left|b_{2}\right|^{2}}{a_{1}}-\frac{\left|a_{1}b_{3}-\overline{b_{1}}b_{2}\right|^{2}}{a_{1}\left(a_{1}a_{2}-\left|b_{1}\right|^{2}\right)},
\end{align*}
and we get the required identity by multiplying the two expressions
for $\mathrm{dist}^{2}\left(v_{3},\mathrm{span}\left\{ v_{1},v_{2}\right\} \right)$
by $a_{1}\left(a_{1}a_{2}-\left|b_{1}\right|^{2}\right)$. By continuity
the identity extends to positive semi-definite matrices.
\end{proof}
Put
\[
g_{j}\left(\theta\right)=\sum_{n\in\bbz}n^{j}c_{n}^{2}e^{in\theta},\quad j\in\left\{ 0,1,2\right\} ,
\]
and write $g$ for $g_{0}$.
\begin{claim}
\label{clm:basic_b_ineq} Assume that $c_{n}\in\bbr$ for all $n\in\bbz$
and that $\sum_{n\in\bbz}n^{2}c_{n}^{2}<\infty$, then
\[
\left|g\left(0\right)g_{1}\left(\theta\right)-g_{1}\left(0\right)g\left(\theta\right)\right|^{2}\le\left(g\left(0\right)g_{2}\left(0\right)-g_{1}^{2}\left(0\right)\right)\left(g^{2}\left(0\right)-\left|g\left(\theta\right)\right|^{2}\right).
\]
\end{claim}

\begin{proof}
By Claim \ref{clm:identity}, and using $\overline{g_{j}\left(\theta\right)}=g_{j}\left(-\theta\right)$,
we have
\begin{align*}
\left|g\left(0\right)g_{1}\left(\theta\right)-g_{1}\left(0\right)g\left(\theta\right)\right|^{2} & =\left(g\left(0\right)g_{2}\left(0\right)-g_{1}^{2}\left(0\right)\right)\left(g^{2}\left(0\right)-\left|g\left(\theta\right)\right|^{2}\right)\\
 & \quad-g\left(0\right)\cdot\det\left(\begin{array}{ccc}
g\left(0\right) & g\left(\theta\right) & g_{1}\left(0\right)\\
g\left(-\theta\right) & g\left(0\right) & g_{1}\left(-\theta\right)\\
g_{1}\left(0\right) & g_{1}\left(\theta\right) & g_{2}\left(0\right)
\end{array}\right).
\end{align*}
Let $h_{j}\left(\theta\right)=\sum_{n\in\bbz}\xi_{n}n^{j}c_{n}e^{in\theta}$
for $j\in\left\{ 0,1\right\} $, put $V=\left(h\left(0\right),h\left(-\theta\right),h_{1}\left(0\right)\right)$
(again with $h=h_{0}$), then we have
\[
\left(\ex{V_{j}\overline{V_{k}}}\right)_{j,k=1}^{3}=\left(\begin{array}{ccc}
g\left(0\right) & g\left(\theta\right) & g_{1}\left(0\right)\\
g\left(-\theta\right) & g\left(0\right) & g_{1}\left(-\theta\right)\\
g_{1}\left(0\right) & g_{1}\left(\theta\right) & g_{2}\left(0\right)
\end{array}\right).
\]
Thus the matrix above is a covariance matrix, hence positive semi-definite
and its determinant is non-negative.
\end{proof}

\subsection{Proof of Theorem \ref{thm:upper_bound}}

Now let $f$ be a Gaussian analytic function with covariance kernel
$G$, whose radius of convergence is $R_{G}$. For the proof again
it will be more convenient to use the exponential parameterization
\[
H\left(z\right)\eqdef G\left(e^{z}\right)=\sum_{n=0}^{\infty}a_{n}^{2}e^{nz},
\]
so that
\[
t_{G}=\log R_{G},\quad A\left(z\right)=\frac{H^{\prime}\left(z\right)}{H\left(z\right)}=\left(\log H\left(z\right)\right)^{\prime},\quad B\left(z\right)=A^{\prime}\left(z\right)=\left(\log H\left(z\right)\right)^{\prime\prime}.
\]
We have (see Claim \ref{clm:var_formula_appendix} in Appendix \ref{sec:Kahane_formulas})
\[
\var{n_{f}\left(r\right)}=\frac{1}{2\pi}\int_{-\pi}^{\pi}\frac{\left|H\left(t\right)H^{\prime}\left(t+i\theta\right)-H\left(t+i\theta\right)H^{\prime}\left(t\right)\right|^{2}}{H^{2}\left(t\right)\left(H^{2}\left(t\right)-\left|H^{2}\left(t+i\theta\right)\right|\right)}\dd\theta,\qquad\text{where }e^{t}=r^{2}.
\]
The following claim will finish the proof of Theorem \ref{thm:upper_bound}.
\begin{cor}
For all $t<t_{G}$ and $\theta\in\left[-\pi,\pi\right]$
\[
\frac{\left|H\left(t\right)H'\left(t+i\theta\right)-H\left(t+i\theta\right)H'\left(t\right)\right|^{2}}{H^{2}(t)\left(H^{2}(t)-\left|H^{2}(t+i\theta)\right|\right)}\le B\left(t\right).
\]
\end{cor}

\begin{proof}
The result follows immediately from Claim \ref{clm:basic_b_ineq},
by taking $c_{n}^{2}=a_{n}^{2}e^{nt}$ for $n\in\bbn$, and $c_{n}=0$
otherwise, so that
\[
g\left(\theta\right)=\sum_{n=0}^{\infty}a_{n}^{2}e^{nt}e^{in\theta}=H\left(t+i\theta\right),\quad g_{1}\left(\theta\right)=H^{\prime}\left(t+i\theta\right),\quad g_{2}\left(\theta\right)=H^{\prime\prime}\left(t+i\theta\right),
\]
and
\[
\frac{g\left(0\right)g_{2}\left(0\right)-g_{1}^{2}\left(0\right)}{g^{2}\left(0\right)}=\frac{H\left(t\right)H^{\prime\prime}\left(t\right)-\left(H^{\prime}\left(t\right)\right)^{2}}{H^{2}\left(t\right)}=\left(\log H\left(t\right)\right)^{\prime\prime}=B\left(t\right).
\]
Notice that for $t<t_{G}$ we have that $\sum_{n\in\bbz}n^{2}c_{n}^{2}=\sum_{n\ge0}n^{2}a_{n}^{2}e^{nt}<\infty$
by the definition of $t_{G}$.
\end{proof}

\section{Gaussian entire functions with large variance - proof of Theorem~\ref{thm:sharpness_upper_bound_example}}

In this section, we will prove Theorem \ref{thm:sharpness_upper_bound_example}.
For an entire function $G$ we recall that
\[
a_{G}\left(r\right)=r\left(\log G\left(r\right)\right)^{\prime},\quad b_{G}\left(r\right)=ra_{G}^{\prime}\left(r\right),
\]
where in this section, we will sometimes add the subscript $G$ in
order to distinguish between functions $a,b$ associated with \emph{different}
entire functions $G$. Given a type II admissible covariance function
$G$ (see Section \ref{subsec:type_II_admissible}), we will construct
a Gaussian entire function $\widetilde{f}$ with a transcendental
covariance function $\widetilde{G}$, which is similar to $G$, that
is
\[
a_{\widetilde{G}}\left(r\right)\asymp_{L}a_{G}\left(r\right),\text{ and }b_{\widetilde{G}}\left(r\right)\asymp_{L}b_{G}\left(r\right),
\]
moreover, $\widetilde{G}$ is a restriction of $G$ (see Definition
\ref{def:restriction}). We then prove
\begin{equation}
\var{n_{f}\left(r\right)}\asymp_{L}b_{\widetilde{G}}\left(r^{2}\right),\label{eq:ex_Var_low_bnd}
\end{equation}
thus showing that the bound in Theorem \ref{thm:upper_bound} is sharp
up to a constant (and an exceptional set of values of $r$).

\subsection{Estimates for sums of Gaussians}

In this section we fix the parameters $A,B,s>0$ and $p$ a positive
integer.
\begin{claim}
\label{clm:Gauss_sum_bnds} Suppose $s\ge1$, and $p\le\sqrt{B}$.
For $j\in\left\{ 0,1,2\right\} $, we have
\[
\frac{1}{\sqrt{e}}\left(A-s\sqrt{B}\right)^{j}\le\sum_{\left|kp-A\right|\le s\sqrt{B}}\left(kp\right)^{j}\exp\left(-\frac{\left(kp-A\right)^{2}}{2B}\right)\le e^{2}\left[\frac{\sqrt{2\pi B}}{p}\left(A^{j}+B\cdot\mathbf{1}_{\left\{ j=2\right\} }\right)+1\right],
\]
where the sum runs over integers $k$, and
\[
\mathbf{1}_{\left\{ j=2\right\} }=\begin{cases}
0, & j\ne2;\\
1, & j=2.
\end{cases}
\]
\end{claim}

\begin{proof}
Since $p\le\sqrt{B}$ the sum is non-empty and the lower bound follows.
In the other direction, put $\phi\left(x\right)=\exp\left(-\frac{\left(x-A\right)^{2}}{2B}\right)$.
If $\left(k-1\right)p\ge A$, then we have
\begin{align*}
\left(kp\right)^{j}\phi\left(kp\right) & \le\frac{1}{p}\left(\frac{k}{k-1}\right)^{j}\int_{\left(k-1\right)p}^{kp}x^{j}\phi\left(x\right)\,\dd x\\
 & \le\frac{1}{p}\exp\left(\frac{j}{k-1}\right)\int_{\left(k-1\right)p}^{kp}x^{j}\phi\left(x\right)\,\dd x\\
 & \le\frac{e^{2}}{p}\int_{\left(k-1\right)p}^{kp}x^{j}\phi\left(x\right)\,\dd x,
\end{align*}
where in the last inequality we used $k\ge2$. Thus,
\[
\sum_{A+p\le kp\le s\sqrt{B}}\left(kp\right)^{j}\phi\left(kp\right)\le\frac{e^{2}}{p}\int_{A}^{A+s\sqrt{B}}x^{j}\phi\left(x\right)\,\dd x.
\]
Using a similar argument for $\left(k-1\right)p\le A$ we get (adding
the integral from $-p$ to $p$ twice)
\begin{align*}
\sum_{\left|kp-A\right|\le s\sqrt{B}}\left(kp\right)^{j}\phi\left(kp\right) & \le\frac{e^{2}}{p}\int_{A-s\sqrt{B}}^{A+s\sqrt{B}}x^{j}\phi\left(x\right)\,\dd x+e^{2}\\
 & \overset{y=\frac{x-A}{\sqrt{B}}}{=}\frac{e^{2}\sqrt{B}}{p}\int_{-s}^{s}\left(y\sqrt{B}+A\right)^{j}e^{-\frac{1}{2}y^{2}}\,\dd y+e^{2}\\
 & \le e^{2}\left[\frac{\sqrt{2\pi B}}{p}\left(A^{j}+B\cdot\mathbf{1}_{\left\{ j=2\right\} }\right)+1\right].
\end{align*}
\end{proof}
\begin{claim}
\label{clm:double_Gauss_bnds} Suppose $s\ge1$, $p\le\sqrt{B}$.
We have,
\[
\frac{p^{2}}{e}\le\sum_{\begin{array}{c}
\left|k_{1}p-A\right|\le s\sqrt{B}\\
\left|k_{2}p-A\right|\le s\sqrt{B}
\end{array}}\left(k_{1}p-k_{2}p\right)^{2}\exp\left(-\frac{\left(k_{1}p-A\right)^{2}}{2B}-\frac{\left(k_{2}p-A\right)^{2}}{2B}\right)\le\frac{24e^{4}\pi B^{2}}{p^{2}},
\]
where the double sum runs over integers $k_{1},k_{2}$.
\end{claim}

\begin{proof}
The lower bound follows since $p\le\sqrt{B}$, so that the double
sum is not empty. Put $\psi\left(x,y\right)=\left(x-y\right)^{2}\phi\left(x\right)\phi\left(y\right)$
so that
\begin{gather*}
\sum_{\begin{array}{c}
\left|k_{1}p-A\right|\le s\sqrt{B}\\
\left|k_{2}p-A\right|\le s\sqrt{B}
\end{array}}\left(pk_{1}-pk_{2}\right)^{2}\exp\left(-\frac{\left(k_{1}p-A\right)^{2}}{2B}-\frac{\left(k_{2}p-A\right)^{2}}{2B}\right)\\
=\sum_{\begin{array}{c}
\left|k_{1}p-A\right|\le s\sqrt{B}\\
\left|k_{2}p-A\right|\le s\sqrt{B}
\end{array}}\psi\left(k_{1}p,k_{2}p\right).
\end{gather*}
Writing $n_{1}=k_{1}+k_{2}$ and $n_{2}=k_{1}-k_{2}$, we find that
\begin{gather}
\sum_{\begin{array}{c}
\left|k_{1}p-A\right|\le s\sqrt{B}\\
\left|k_{2}p-A\right|\le s\sqrt{B}
\end{array}}\psi\left(k_{1}p,k_{2}p\right)\le\sum_{\begin{array}{c}
\left|n_{2}p\right|\le2s\sqrt{B}\\
\left|n_{1}p-2A\right|\le2s\sqrt{B}
\end{array}}\psi\left(\frac{\left(n_{2}-n_{1}\right)p}{2},\frac{\left(n_{2}+n_{1}\right)p}{2}\right)\nonumber \\
=\sum_{\left|n_{2}p\right|\le2s\sqrt{B}}\left(n_{2}p\right)^{2}\exp\left(-\frac{\left(n_{2}p\right)^{2}}{4B}\right)\cdot\sum_{\left|n_{1}p-2A\right|\le2s\sqrt{B}}\exp\left(-\frac{\left(n_{1}p-2A\right)^{2}}{4B}\right),\label{eq:double_sum}
\end{gather}
where we used the following identity in the last equality
\[
\psi\left(\frac{x-y}{2},\frac{x+y}{2}\right)=y^{2}\exp\left(-\frac{y^{2}}{4B}\right)\exp\left(-\frac{\left(x-2A\right)^{2}}{4B}\right).
\]
Bounding both factors in (\ref{eq:double_sum}) using the previous
claim, we conclude that
\begin{align*}
\sum_{\begin{array}{c}
\left|k_{1}p-A\right|\le s\sqrt{B}\\
\left|k_{2}p-A\right|\le s\sqrt{B}
\end{array}}\psi\left(k_{1}p,k_{2}p\right) & \le e^{4}\left(\frac{\sqrt{2\pi\cdot2B}}{p}\cdot2B+1\right)\left(\frac{\sqrt{2\pi\cdot2B}}{p}+1\right)\\
 & \le e^{4}\left(\frac{\sqrt{4\pi B}}{p}\right)^{2}\cdot\left(2B+1\right)\cdot2\\
 & \le\frac{24e^{4}\pi B^{2}}{p^{2}}.
\end{align*}
\end{proof}

\subsection{\label{subsec:type_II_admissible} Type II admissible functions and
their Taylor coefficients}

As before, we associate with $G$ the functions
\[
H\left(t\right)=G\left(e^{t}\right),\quad A\left(t\right)=\frac{H^{\prime}\left(t\right)}{H\left(t\right)},\quad B\left(t\right)=A^{\prime}\left(t\right)=\frac{H^{\prime\prime}\left(t\right)}{H\left(t\right)}-\left(\frac{H^{\prime}\left(t\right)}{H\left(t\right)}\right)^{2}.
\]
In this section we will use a stronger version of the Hayman admissibility
condition, which allows us to improve the error term in \cite[Theorem I]{Hayman}
(see the lemma below).
\begin{defn}
We say that an \emph{entire} function $G$ is \emph{type II} \emph{admissible}
if the function $H$ has the following properties: 
\end{defn}

\begin{enumerate}
\item \label{enu:Hayman_Admis_Assum_growth_B-1} $B\uparrow\infty$
\item \label{enu:Assump_A_control_by_B-1} $A\left(t\right)=O\left(B^{2}\left(t\right)\right)$
as $t\to\infty$.
\item \label{enu:Hayman_Admis_Assum_Asymp-1} There exist constants $C_{G}>2$
and $\varepsilon>0$ such that
\[
\log\frac{H\left(t+i\theta\right)}{H\left(t\right)}=i\theta A\left(t\right)-\frac{1}{2}\theta^{2}B\left(t\right)+\Delta\left(t,\theta\right),\qquad\forall\left|\theta\right|\le\delta\left(t\right),
\]
where
\[
\delta\left(t\right):=\sqrt{C_{G}\frac{\log B\left(t\right)}{B\left(t\right)}},\qquad\left|\Delta\left(t,\theta\right)\right|\le B^{3/2-\varepsilon}\left(t\right)\left|\theta\right|^{3}.
\]
\item \label{enu:Hayman_Admis_Assum_decay-1} $\left|H\left(t+i\theta\right)\right|=O\left(\frac{H\left(t\right)}{B\left(t\right)}\right)$
as $t\to\infty$, uniformly in $\left|\theta\right|\in\left[\delta\left(t\right),\pi\right]$.
\end{enumerate}
We again put $r^{2}=e^{t}$.
\begin{rem}
Throughout this section, in order make the expressions shorter, we
may suppress the dependence of $H,A,B,$ and other parameters on $t$
inside the proofs. 
\end{rem}

\begin{lem}[{cf. \cite[Theorem 1]{Hayman}}]
\label{lem:coeff_asymp} Suppose that $G\left(z\right)=\sum_{n=0}^{\infty}a_{n}^{2}z^{n}$
is type II admissible, then there exists an $\varepsilon\in\left(0,1\right)$
such that for all $t$ sufficiently large,
\[
a_{n}^{2}e^{nt}=\frac{H\left(t\right)}{\sqrt{2\pi B\left(t\right)}}\left[\exp\left(-\frac{\left(n-A\left(t\right)\right)^{2}}{2B\left(t\right)}\right)+O\left(\frac{1}{B^{\varepsilon}\left(t\right)}\right)\right],\qquad\text{as \ensuremath{\quad t\to\infty,}}
\]
uniformly in $n$.
\end{lem}

\begin{proof}
Put $\delta=\delta\left(t\right)$. By Cauchy's formula
\[
a_{n}^{2}e^{nt}=\frac{1}{2\pi}\int_{-\pi}^{\pi}H\left(t+i\theta\right)e^{-in\theta}\,\dd\theta,
\]
which we write as $a_{n}^{2}e^{nt}=I_{1}+I_{2}$, where
\[
I_{1}=\frac{1}{2\pi}\int_{\left|\theta\right|\le\delta}H\left(t+i\theta\right)e^{-in\theta}\,\dd\theta,\qquad I_{2}=\frac{1}{2\pi}\int_{\delta<\left|\theta\right|\le\pi}H\left(t+i\theta\right)e^{-in\theta}\,\dd\theta.
\]
By Assumption \ref{enu:Hayman_Admis_Assum_decay-1}, we have, uniformly
in $n$,
\[
I_{2}=O\left(\frac{H\left(t\right)}{B\left(t\right)}\right).
\]
By Assumption  we have
\begin{align*}
I_{1} & =\frac{H}{2\pi}\int_{\left|\theta\right|\le\delta}\exp\left(-i\theta\left(n-A\right)-\frac{1}{2}\theta^{2}B+\Delta\left(\theta\right)\right)\,\dd\theta\\
 & =\frac{H}{2\pi}\int_{\left|\theta\right|\le\delta}\left(1+O\left(\Delta\left(\theta\right)\right)\right)\exp\left(-i\theta\left(n-A\right)-\frac{1}{2}\theta^{2}B\right)\,\dd\theta\\
 & =\frac{H}{2\pi}\int_{\left|\theta\right|\le\delta}\exp\left(-i\theta\left(n-A\right)-\frac{1}{2}\theta^{2}B\right)\,\dd\theta+O\left(B^{3/2-\varepsilon}\delta^{4}\right)\\
 & =\frac{H}{2\pi}\int_{\left|\theta\right|\le\delta}\exp\left(-i\theta\left(n-A\right)-\frac{1}{2}\theta^{2}B\right)\,\dd\theta+O\left(\frac{\left(\log B\right)^{2}}{B^{\frac{1}{2}+\varepsilon}}\right).
\end{align*}
In order to find the asymptotic behavior of $I_{1}$, we make the
change of variables
\[
y=\theta\sqrt{\frac{B}{2}},\qquad\alpha\eqdef\frac{n-A}{\sqrt{B/2}},
\]
and obtain (uniformly in $n$)
\begin{align*}
I_{1} & =\frac{H}{\pi\sqrt{2B}}\int_{\left|y\right|\le\delta\sqrt{\frac{B}{2}}}\exp\left(-y^{2}-i\alpha y\right)\,\dd y+O\left(\frac{1}{B^{\frac{1}{2}\left(1+\varepsilon\right)}}\right)\\
 & =\frac{H}{\pi\sqrt{2B}}\left[\int_{\bbr}-\int_{\left|y\right|\ge\delta\sqrt{\frac{B}{2}}}\right]\exp\left(-y^{2}-i\alpha y\right)\,\dd y+O\left(\frac{1}{B^{\frac{1}{2}\left(1+\varepsilon\right)}}\right)\\
 & =\frac{H}{\sqrt{2\pi B}}\exp\left(-\frac{1}{4}\alpha^{2}\right)+O\left(\exp\left(-\left(\delta\sqrt{\frac{B}{2}}\right)^{2}\right)\right)+O\left(\frac{1}{B^{\frac{1}{2}\left(1+\varepsilon\right)}}\right)\\
 & =\frac{H}{\sqrt{2\pi B}}\exp\left(-\frac{1}{4}\alpha^{2}\right)+O\left(\frac{1}{B}\right)+O\left(\frac{1}{B^{\frac{1}{2}\left(1+\varepsilon\right)}}\right)\\
 & =\frac{H}{\sqrt{2\pi B}}\exp\left(-\frac{\left(n-A\right)^{2}}{2B}\right)+O\left(\frac{1}{B^{\frac{1}{2}\left(1+\varepsilon\right)}}\right).
\end{align*}
\end{proof}

\subsection{Construction of the covariance function $\widetilde{G}$}

Recall that in Section \ref{subsec:normal_values} we defined a sequence
of intervals $\left\{ T_{\ell}\right\} _{\ell=1}^{\infty}$ so that
\[
T_{\ell}=\left[t_{\ell},t_{\ell+1}\right],\qquad\left|T_{\ell}\right|=t_{\ell+1}-t_{\ell},\qquad B\left(t_{\ell}\right)=\ell^{6},\qquad\ell\ge1.
\]
Moreover, the interval $T_{\ell}$ is \emph{long} if $\left|T_{\ell}\right|\ge\frac{8}{\ell^{2}},$
and in that case its \emph{interior} is given by
\[
\longint{\ell}=\left[t_{\ell}+\frac{2}{l^{2}},t_{\ell+1}-\frac{2}{\ell^{2}}\right].
\]
The set $\cX$ of normal values is given by
\[
\cX\eqdef\bigcup_{T_{\ell}\text{ long}}\longint{\ell}.
\]
We also remind that the Lebesgue measure of the set $\bbr^{+}\backslash\cX$
is finite.

We now fix the sequences $p_{\ell}=\ell^{3}=\sqrt{B\left(t_{\ell}\right)}$
, and $s_{\ell}=c_{1}\sqrt{\log\ell}$ (see Proposition \ref{prop:Q_bnds})
and define the following sets
\[
\cI_{\ell}\eqdef\left[A\left(t_{\ell}\right),A\left(t_{\ell+1}\right)\right)\cap\bbn,\qquad\cI_{\ell}^{\mathbf{p}}\eqdef\left\{ n\in\cI_{\ell}\,:\,p_{\ell}\mid n\right\} ,
\]
and the corresponding exponential polynomials
\begin{equation}
P_{\ell}\left(t\right)=\sum_{n\in\cI_{\ell}^{\mathbf{p}}}a_{n}^{2}e^{nt}.\label{eq:P_l_def}
\end{equation}
The function $\widetilde{G}$ is constructed as follows

\begin{equation}
\widetilde{G}\left(e^{t}\right)=\widetilde{H}\left(t\right)=\sum_{\ell=0}^{\infty}P_{\ell}\left(t\right)\defeq\sum_{n=0}^{\infty}\delta_{n}a_{n}^{2}e^{nt}.\label{eq:G_tilde_def}
\end{equation}
For $t\in\longint{\ell}$ the main indices $n$ corresponding to $t$
are included in the following window
\[
\cI_{\ell}^{\mathbf{p}}\left(t\right)\eqdef\left\{ n\in\cI_{\ell}^{\mathbf{p}}\,:\,\left|n-A\left(t\right)\right|\le s_{\ell}\sqrt{B\left(t\right)}\right\} .
\]
Finally, we put
\[
R_{\ell}^{\left(j\right)}\left(t\right)=\sum_{n\in\cI_{\ell}^{\mathbf{p}}\left(t\right)}n^{j}a_{n}^{2}e^{nt},\qquad\overline{R}_{\ell}^{\left[j\right]}\left(t\right)=\sum_{n\in\cI_{\ell}^{\mathbf{p}}\left(t\right)}\left(n-A\left(t\right)\right)^{j}a_{n}^{2}e^{nt},
\]
and
\[
Q_{\ell}^{\left(j\right)}\left(\tau,t\right)=\sum_{\left|n-A\left(t\right)\right|>s_{\ell}\sqrt{B\left(t\right)}}n^{j}a_{n}^{2}e^{n\tau},\qquad\overline{Q}_{\ell}^{\left[j\right]}\left(\tau,t\right)=\sum_{\left|n-A\left(t\right)\right|>s_{\ell}\sqrt{B\left(t\right)}}\left(n-A\left(t\right)\right)^{j}a_{n}^{2}e^{n\tau}.
\]

\begin{prop}
\label{prop:Q_bnds} For $t\in\cX$ and $\left|\tau-t\right|\le\frac{1}{2\sqrt{B\left(t\right)}}$,
we have
\[
\max_{j\in\left\{ 0,1,2\right\} }\left\{ \left|Q_{\ell}^{\left(j\right)}\left(\tau\right)\right|,\left|\overline{Q}_{\ell}^{\left[j\right]}\left(\tau,t\right)\right|\right\} \le\frac{H\left(\Re{\tau}\right)}{B^{3}\left(t\right)}.
\]
\end{prop}

\begin{rem}
The choice of the exponent $3$ is arbitrary, it can be replaced by
any large positive number.
\end{rem}

\begin{proof}
Fix an interval $\longint{\ell}\subset\cX$. By Lemma \ref{lem:poly_approx}
we will have
\[
\left|Q_{\ell}\left(\tau\right)\right|\le\sum_{\left|k-A\left(t\right)\right|>s\sqrt{B\left(t\right)}}a_{k}^{2}e^{k\text{\ensuremath{\Re{\ensuremath{\tau}}}}}\le2H\left(\Re{\tau}\right)\exp\left(-\frac{1}{8}\left(s-4\right)^{2}\right),
\]
as long as $4<s<B^{1/6}\left(t_{\ell}\right)$. Notice that in order
to guarantee that
\[
2\exp\left(-\frac{1}{8}\left(s-4\right)^{2}\right)\le\frac{1}{\ell^{48}}=\frac{1}{B^{8}\left(t_{\ell}\right)}\le\frac{1}{B^{8}\left(t\right)},
\]
it is sufficient to take $s=4+2\sqrt{2\log2+96\log\left(\ell\right)}$.
Thus, we may choose any $c_{1}>8\sqrt{6}$ in the definition of $s_{\ell}$,
and obtain

\[
\left|Q_{\ell}\left(\tau\right)\right|\le\frac{H\left(\Re{\tau}\right)}{B^{8}\left(t\right)}.
\]
 For $\tau\in\bbc$ which satisfies $\left|\tau-t\right|<\frac{1}{2\sqrt{B\left(t\right)}}$,
consider the contour $\Gamma=\left\{ z:\left|z-\tau\right|=\frac{1}{2\sqrt{B\left(t\right)}}\right\} $.
Using Cauchy's integral formula, we get
\[
\left|Q_{\ell}^{\left(j\right)}\left(\tau\right)\right|=\left|\frac{1}{2\pi i}\int_{\Gamma}\frac{Q_{\ell}\left(z\right)}{\left(z-\tau\right)^{j+1}}\dd z\right|\lesssim\frac{H\left(\Re{\tau}\right)}{\left[B\left(t\right)\right]^{8-\frac{1}{2}j}}.
\]
In addition, since the the coefficients $a_{n}^{2}$ are non-negative,
we also notice that when $\left|\Re{\tau}-t\right|\le\frac{1}{2\sqrt{B\left(t\right)}}$,
and $j\in\left\{ 0,1,2\right\} $,
\begin{equation}
\left|Q_{\ell}^{\left(j\right)}\left(\tau\right)\right|\le\left|Q_{\ell}^{\left(j\right)}\left(\Re{\tau}\right)\right|\le\frac{H\left(\Re{\tau}\right)}{B^{7}\left(t\right)}.\label{eq:bound_for_Q_ell_deriv}
\end{equation}
To bound $\left|\overline{Q}_{\ell}^{\left[j\right]}\left(\tau,t\right)\right|$,
notice that $\overline{Q}_{\ell}^{\left[0\right]}\left(\tau,t\right)=Q_{\ell}\left(\tau,t\right)$,
\begin{align*}
\overline{Q}_{\ell}^{\left[1\right]}\left(\tau,t\right) & =Q_{\ell}^{\prime}\left(\tau,t\right)-A\left(t\right)\cdot Q_{\ell}\left(\tau,t\right),\\
\overline{Q}_{\ell}^{\left[2\right]}\left(\tau,t\right) & =Q_{\ell}^{\prime\prime}\left(\tau,t\right)-2A\left(t\right)\cdot Q_{\ell}^{\prime}\left(\tau,t\right)+A\left(t\right)^{2}\cdot Q_{\ell}\left(\tau,t\right),
\end{align*}
and use Assumption \ref{enu:Assump_A_control_by_B-1}.
\end{proof}
\begin{prop}
\label{prop:P_bnds} There are constants $c,C>0$ such that for $t\in\cX$,
and $j\in\left\{ 0,1,2\right\} $, we have
\[
c\frac{H\left(t\right)\left(A\left(t\right)\right)^{j}}{\sqrt{B\left(t\right)}}\le R_{\ell}^{\left(j\right)}\left(t\right)\le C\frac{H\left(t\right)\left(A\left(t\right)\right)^{j}}{\sqrt{B\left(t\right)}},
\]
and
\[
\overline{R}_{\ell}^{\left[j\right]}\left(t\right)\le C\frac{H\left(t\right)\left(s_{\ell}\sqrt{B\left(t\right)}\right)^{j}}{\sqrt{B\left(t\right)}}.
\]
\end{prop}

\begin{proof}
By Lemma \ref{lem:coeff_asymp}, Claim \ref{clm:Gauss_sum_bnds},
and (\ref{eq:b_bound_in_a}) we obtain
\begin{align*}
R_{\ell}^{\left(j\right)}\left(t\right) & =\sum_{\left|kp_{\ell}-A\right|<s_{\ell}\sqrt{B}}\left(kp\right)^{j}a_{kp}^{2}e^{kpt}\\&=\frac{H}{\sqrt{2\pi B}}\sum_{\left|kp_{\ell}-A\right|<s_{\ell}\sqrt{B}}\left(kp\right)^{j}\left[\exp\left(-\frac{\left(kp-A\right)^{2}}{2B}\right)+O\left(\frac{1}{B^{\varepsilon}}\right)\right]\\
 & \lesssim\frac{H}{\sqrt{B}}\left[\frac{\sqrt{B}}{p_{\ell}}\left(A^{j}+B\cdot\mathbf{1}_{\left\{ j=2\right\} }\right)+1+\frac{s_{\ell}\sqrt{B}}{p_{\ell}}\cdot\frac{\left(A+s_{\ell}\sqrt{B}\right)^{j}}{B^{\varepsilon}}\right]\\
 & \lesssim\frac{H}{p_{\ell}}\cdot A{}^{j}\left[1+s_{\ell}\cdot\frac{1}{B^{\varepsilon}}\right] \lesssim\frac{H\left(t\right)}{\sqrt{B\left(t\right)}}\cdot A\left(t\right){}^{j}\left[1+s_{\ell}\cdot\frac{1}{B^{\varepsilon}\left(t\right)}\right].
\end{align*}
The lower bound is obtained in a similar way, using the lower bound
in Claim \ref{clm:Gauss_sum_bnds}. To bound $\overline{R}_{\ell}^{\left[j\right]}\left(t\right)$,
we simply use
\[
\overline{R}_{\ell}^{\left[j\right]}\left(t\right)\le\sum_{n\in\cI_{\ell}^{\mathbf{p}}\left(t\right)}\left|n-A\left(t\right)\right|^{j}a_{n}^{2}e^{nt}\le\left(s_{\ell}\sqrt{B}\right)^{j}\overline{R}_{\ell}\left(t\right)\lesssim\frac{H\left(t\right)\left(s_{\ell}\sqrt{B}\right)^{j}}{\sqrt{B\left(t\right)}}.
\]
\end{proof}
Now we are ready to prove that $G$ and $\widetilde{G}$ are similar.
\begin{lem}
\label{lem:A_and_B_H_tilde_new} We have
\[
a_{\widetilde{H}}\left(r\right)\asymp_{L}a\left(r\right),\qquad b_{\widetilde{H}}\left(r\right)\asymp_{L}b\left(r\right).
\]
\end{lem}

\begin{proof}
Let $c,C>0$ denote constants. Fix $t\in\cX$, it is sufficient to
show that
\[
cA\left(t\right)\le A_{\widetilde{H}}\left(t\right)\le CA\left(t\right),\qquad cB\left(t\right)\le B_{\widetilde{H}}\left(t\right)\le CB\left(t\right).
\]
It follows from the identities
\[
A_{\widetilde{H}}\left(t\right)=\frac{\widetilde{H}^{\prime}\left(t\right)}{\widetilde{H}\left(t\right)},\quad B_{\widetilde{H}}\left(t\right)=A_{\widetilde{H}}^{\prime}\left(t\right)=\frac{\widetilde{H}^{\prime\prime}\left(t\right)}{\widetilde{H}\left(t\right)}-\left(\frac{\widetilde{H}^{\prime}\left(t\right)}{\widetilde{H}\left(t\right)}\right)^{2},
\]
and the definition of $\widetilde{H}$ in (\ref{eq:G_tilde_def}),
that
\[
A_{\widetilde{H}}\left(t\right)=\frac{\sum_{n=0}^{\infty}n\delta_{n}a_{n}^{2}e^{nt}}{\widetilde{H}\left(t\right)},\qquad B_{\widetilde{H}}\left(t\right)=\frac{\sum_{n,m=0}^{\infty}\left(n-m\right)^{2}\delta_{n}\delta_{m}a_{n}^{2}a_{m}^{2}e^{\left(n+m\right)t}}{2\left(\widetilde{H}\left(t\right)\right)^{2}}.
\]
Notice that for $j\in\left\{ 0,1,2\right\} $ we have
\begin{equation}
\left|\widetilde{H}^{\left(j\right)}\left(t\right)-R_{\ell}^{\left(j\right)}\left(t\right)\right|\leq Q_{\ell}^{\left(j\right)}\left(t,t\right).\label{eq:H_tilde_j_eq}
\end{equation}
Therefore, by Propositions \ref{prop:Q_bnds} and \ref{prop:P_bnds}
we have
\[
A_{\widetilde{H}}\left(t\right)=\text{\ensuremath{\frac{R_{\ell}^{\prime}\left(t\right)+\left(H^{\prime}\left(t\right)-R_{\ell}^{\prime}\left(t\right)\right)}{R_{\ell}\left(t\right)+\left(H\left(t\right)-R_{\ell}\left(t\right)\right)}\lesssim\frac{C\frac{H\left(t\right)A\left(t\right)}{\sqrt{B\left(t\right)}}+\frac{H\left(t\right)}{B^{3}\left(t\right)}}{c\frac{H\left(t\right)}{\sqrt{B\left(t\right)}}-\frac{H\left(t\right)}{B^{3}\left(t\right)}}\lesssim A\left(t\right),}}
\]
and similarly for the lower bound. Put
\[
\Sigma_{1}\eqdef\sum_{n,m\in\cI_{\ell}^{\mathbf{p}}\left(t\right)}\left(n-m\right)^{2}a_{n}^{2}a_{m}^{2}e^{\left(n+m\right)t},
\]
and notice that clearly
\[
\Sigma_{1}\le2B_{\widetilde{H}}\left(t\right)\left(\widetilde{H}\left(t\right)\right)^{2}.
\]
In addition,
\begin{align*}
2B_{\widetilde{H}}\left(t\right)\left(\widetilde{H}\left(t\right)\right)^{2} & =\sum_{n,m=0}^{\infty}\left(n-m\right)^{2}\delta_{n}\delta_{m}a_{n}^{2}a_{m}^{2}e^{\left(n+m\right)t}\le\sum_{n,m\in\cI_{\ell}^{\mathbf{p}}\left(t\right)}\left(n-m\right)^{2}a_{n}^{2}a_{m}^{2}e^{\left(n+m\right)t}\\
 & +2\cdot\sum_{n\in\cI_{\ell}^{\mathbf{p}}\left(t\right),\,\left|m-A\left(t\right)\right|>s_{\ell}\sqrt{B\left(t\right)}}\left(n-m\right)^{2}a_{n}^{2}a_{m}^{2}e^{\left(n+m\right)t}\\
 & +\sum_{\begin{array}{c}
\left|n-A\left(t\right)\right|>s_{\ell}\sqrt{B\left(t\right)}\\
\left|m-A\left(t\right)\right|>s_{\ell}\sqrt{B\left(t\right)}
\end{array}}\left(n-m\right)^{2}a_{n}^{2}a_{m}^{2}e^{\left(n+m\right)t}\\
 & \defeq\Sigma_{1}+\Sigma_{2}+\Sigma_{3}.
\end{align*}
By Claim \ref{clm:Sigma_1_bounds} below, we have $cH^{2}\left(t\right)\leq\Sigma_{1}\leq CH^{2}\left(t\right),$
so that again by (\ref{eq:H_tilde_j_eq}) and Propositions \ref{prop:Q_bnds}
and \ref{prop:P_bnds} we have
\[
B\left(t\right)\widetilde{H}^{2}\left(t\right)\lesssim\Sigma_{1}\lesssim B\left(t\right)\widetilde{H}^{2}\left(t\right).
\]
Writing
\[
\left(n-m\right)^{2}\le\frac{1}{2}\left(\left(n-A\right)^{2}+\left(m-A\right)^{2}\right),
\]
we get, using Proposition \ref{prop:Q_bnds},
\begin{align*}
\Sigma_{3} & \le\left(\sum_{\left|n-A\left(t\right)\right|>s_{\ell}\sqrt{B\left(t\right)}}\left(n-A\right)^{2}a_{n}^{2}e^{nt}\right)\left(\sum_{\left|m-A\left(t\right)\right|>s_{\ell}\sqrt{B\left(t\right)}}a_{m}^{2}e^{mt}\right)=\overline{Q}_{\ell}^{\left[2\right]}\left(t,t\right)\overline{Q}_{\ell}^{\left[0\right]}\left(t,t\right)\\
 & \le\frac{H^{2}\left(t\right)}{B^{6}\left(t\right)}\lesssim\frac{\widetilde{H}^{2}\left(t\right)}{B^{5}\left(t\right)}.
\end{align*}
In addition, combining Propositions \ref{prop:Q_bnds} and \ref{prop:P_bnds},
we have
\begin{align*}
\Sigma_{2}\le & \left(\sum_{n\in\cI_{\ell}^{\mathbf{p}}\left(t\right)}\left(n-A\right)^{2}a_{n}^{2}e^{nt}\right)\left(\sum_{\left|m-A\left(t\right)\right|>s_{\ell}\sqrt{B\left(t\right)}}a_{m}^{2}e^{mt}\right)\\
 & +\left(\sum_{n\in\cI_{\ell}^{\mathbf{p}}\left(t\right)}a_{n}^{2}e^{nt}\right)\left(\sum_{\left|m-A\left(t\right)\right|>s_{\ell}\sqrt{B\left(t\right)}}\left(m-A\right)^{2}a_{m}^{2}e^{mt}\right)\\
= & \overline{R}_{\ell}^{\left[2\right]}\left(t\right)\overline{Q}_{\ell}^{\left[0\right]}\left(t,t\right)+\overline{R}_{\ell}^{\left[0\right]}\left(t\right)\overline{Q}_{\ell}^{\left[2\right]}\left(t,t\right)\\
\lesssim & \frac{H\left(t\right)\cdot s_{\ell}^{2}B\left(t\right)}{\sqrt{B\left(t\right)}}\cdot\frac{H\left(t\right)}{B\left(t\right)^{3}}+\frac{H\left(t\right)}{\sqrt{B\left(t\right)}}\cdot\frac{H\left(t\right)}{B\left(t\right)^{3}}\\
\lesssim & \frac{H^{2}\left(t\right)}{B^{2}\left(t\right)}\lesssim\frac{\widetilde{H}^{2}\left(t\right)}{B\left(t\right)}.
\end{align*}
Thus,
\[
B\left(t\right)\lesssim B_{\widetilde{H}}\left(t\right)\lesssim B\left(t\right).
\]
\end{proof}
\begin{claim}
\label{clm:Sigma_1_bounds} There are constants $c,C>0$, so that
for $t\in\cX$, we have
\[
cH^{2}\left(t\right)\leq\Sigma_{1}\leq CH^{2}\left(t\right),
\]
where
\[
\Sigma_{1}\eqdef\sum_{n,m\in\cI_{\ell}^{\mathbf{p}}\left(t\right)}\left(n-m\right)^{2}a_{n}^{2}a_{m}^{2}e^{\left(n+m\right)t}.
\]
\end{claim}

\begin{proof}
By Lemma \ref{lem:coeff_asymp} we have
\[
\Sigma_{1}=\frac{H^{2}}{2\pi B}\sum_{n,m\in\cI_{\ell}^{\mathbf{p}}\left(t\right)}\left(n-m\right)^{2}\left(e^{-\frac{\left(n-A\right)^{2}}{2B}}+O\left(\frac{1}{B^{\varepsilon}}\right)\right)\left(e^{-\frac{\left(m-A\right)^{2}}{2B}}+O\left(\frac{1}{B^{\varepsilon}}\right)\right),
\]
and therefore
\[
\left|\Sigma_{1}-S_{1}\right|\le S_{2},
\]
where
\begin{align*}
S_{1} & \eqdef\frac{H^{2}}{B}\sum_{n,m\in\cI_{\ell}^{\mathbf{p}}\left(t\right)}\left(n-m\right)^{2}\exp\left(-\frac{\left(n-A\right)^{2}}{2B}-\frac{\left(m-A\right)^{2}}{2B}\right),\\
S_{2} & \eqdef\frac{H^{2}}{B^{1+\varepsilon}}\sum_{n,m\in\cI_{\ell}^{\mathbf{p}}\left(t\right)}\left(n-m\right)^{2}.
\end{align*}
By Claim \ref{clm:double_Gauss_bnds}, we have
\[
S_{1}\asymp\frac{H^{2}}{B}\cdot B=H^{2}.
\]
For $n,m\in\cI_{\ell}^{\mathbf{p}}\left(t\right)$, $\left|n-m\right|\leq2s_{\ell}\sqrt{B},$
thus we have
\[
S_{2}\lesssim\frac{H^{2}}{B^{\varepsilon}}\sum_{n,m\in\cI_{\ell}^{\mathbf{p}}\left(t\right)}s_{\ell}^{2}\lesssim\text{\ensuremath{\frac{H^{2}}{B^{\varepsilon}}}}\cdot s_{\ell}^{4}=o\left(H^{2}\right).
\]
\end{proof}

\subsection{Proof of Theorem~\ref{thm:sharpness_upper_bound_example}}

Let
\[
\widetilde{f}\left(z\right)=\sum_{\ell=0}^{\infty}\sum_{n\in\cI_{\ell}^{\mathbf{p}}}\xi_{n}a_{n}z^{n}
\]
be a Gaussian entire function with covariance function $\widetilde{G}$.
In order to prove that $\var{n\left(r\right)}\gtrsim p_{\ell}^{2}\gtrsim B\left(t\right)$
it is enough to show that $n\left(r\right)=kp_{\ell}$ for two different
values of $k$, with probability at least $c>0$ each. By Rouché's
theorem, it is enough to show that the term $\xi_{kp_{\ell}}a_{kp_{\ell}}z^{kp_{\ell}}$
dominates all other terms.
\begin{prop}
There is a constant $c>0$, so that for every $t\in\cX$ the probability
of the event $\left\{ n_{\widetilde{f}}\left(r\right)=kp_{\ell}\right\} $
is at least $c$ for two different values $k\in\bbn$.
\end{prop}

\begin{proof}
Put
\[
m_{1}=\max\left\{ m\in\cI_{\ell}^{\mathbf{p}}:m<A\right\} ,\qquad m_{2}=\min\left\{ m\in\cI_{\ell}^{\mathbf{p}}:m>A\right\} ,
\]
and notice that
\[
\frac{1}{2}p_{\ell}\le\max\left\{ \left|m_{1}-A\right|,\left|m_{2}-A\right|\right\} \le p_{\ell},\qquad p_{\ell}\le\left|m_{1}-m_{2}\right|\le2p_{\ell}.
\]
Define the events
\[
E_{j}=\left\{ \left|\xi_{m_{j}}a_{m_{j}}z^{m_{j}}\right|>\left|f\left(z\right)-\xi_{m_{j}}a_{m_{j}}z^{m_{j}}\right|\right\} ,\qquad j\in\left\{ 1,2\right\} .
\]
We will prove $\pr{E_{1}}>c$, the proof for $E_{2}$ is similar.
The result will then follow by Rouché's theorem. We have the crude
bound:
\[
\bbe\left|f\left(z\right)-\xi_{m_{j}}a_{m_{j}}z^{m_{j}}\right|\le\sum_{n\ne m_{1}}\bbe\left|\xi_{n}\right|\delta_{n}a_{n}r^{n}=\frac{\sqrt{\pi}}{2}\left[S_{1}+S_{2}\right],
\]
where
\[
S_{1}=\sum_{n\ne m_{1},n\in\cI_{\ell}^{\mathbf{p}}}a_{n}r^{n},\qquad S_{2}=\sum_{\left|n-A\right|>s_{\ell}\sqrt{B}}\delta_{n}a_{n}r^{n}.
\]
 By Lemma \ref{lem:coeff_asymp}, we have, uniformly in $n$,
\[
a_{n}^{2}e^{nt}=\frac{H}{\sqrt{2\pi B}}\left[\exp\left(-\frac{\left(n-A\right)^{2}}{2B}\right)+O\left(\frac{1}{B^{\varepsilon}}\right)\right].
\]
Since $\sqrt{x+y}\le\sqrt{x}+\sqrt{y}$, we get by Claim \ref{clm:Gauss_sum_bnds},
\begin{align*}
S_{1} & \le\sum_{n\in\cI_{\ell}^{\mathbf{p}}}a_{n}e^{\frac{1}{2}tn}\lesssim\frac{\sqrt{H}}{B^{1/4}}\sum_{n\in\cI_{\ell}^{\mathbf{p}}}\exp\left(-\frac{\left(n-A\right)^{2}}{4B}\right)+\frac{\sqrt{H}}{B^{1/4}}\cdot\frac{s_{\ell}\sqrt{B}}{p_{\ell}}\cdot O\left(\frac{1}{B^{\varepsilon}}\right)\\
 & \lesssim\frac{\sqrt{H}}{B^{1/4}}\cdot\frac{\sqrt{B}}{p_{\ell}}+O\left(\frac{\sqrt{H}}{B^{1/4+\varepsilon}}\right)\lesssim\frac{\sqrt{H\left(t\right)}}{B^{1/4}\left(t\right)}.
\end{align*}
Now let $\eta=\frac{1}{A}\le\frac{1}{2\sqrt{B}}$, where the inequality
holds (for $t$ sufficiently large) by (\ref{eq:b_bound_in_a}). Writing,
\[
r=e^{\frac{1}{2}t},\quad\tau=t+\eta,\qquad r^{n}=e^{\frac{1}{2}tn}=e^{\frac{1}{2}\left(\tau-\eta\right)n}=e^{\frac{1}{2}\tau n}e^{-\frac{1}{2}\eta n},
\]
and using the Cauchy-Schwarz inequality, Proposition \ref{prop:Q_bnds},
and (\ref{eq:gen_func_est}), we have
\begin{align*}
S_{2}^{2} & \le\sum_{\left|n-A\right|>s_{\ell}\sqrt{B}}a_{n}^{2}e^{\tau n}\cdot\sum_{\left|n-A\right|>s_{\ell}\sqrt{B}}e^{-\eta n}\le Q_{\ell}^{\left(0\right)}\left(\tau,t\right)\cdot\sum_{n=0}^{\infty}e^{-\eta n}\\
 & \le\frac{H\left(\Re{\tau}\right)}{B^{3}}\cdot\frac{2}{\eta}\lesssim H\left(t\right)\exp\left(\eta\cdot A+C\eta^{2}B\right)\cdot\frac{A}{B^{3}}\\
 & \lesssim H\cdot\exp\left(\frac{CB}{A^{2}}\right)\cdot\frac{A}{B^{3}}\lesssim H\cdot\frac{A}{B^{3}},
\end{align*}
where we again used (\ref{eq:b_bound_in_a}) in the last inequality.
Therefore, by Assumption \ref{enu:Assump_A_control_by_B-1} we get
\[
S_{2}\lesssim\frac{\sqrt{H\left(t\right)}}{B^{1/4}\left(t\right)}.
\]
We conclude by Markov's inequality that for $C>0$ sufficiently large,
we have
\[
\pr{\left|f\left(z\right)-\xi_{m_{j}}a_{m_{j}}z^{m_{j}}\right|>\frac{\sqrt{\pi}}{2}\left[S_{1}+S_{2}\right]}\le\pr{\left|f\left(z\right)-\xi_{m_{j}}a_{m_{j}}z^{m_{j}}\right|>\frac{C\sqrt{H\left(t\right)}}{B^{1/4}\left(t\right)}}<\frac{1}{2}.
\]
Finally, again by Lemma \ref{lem:coeff_asymp}
\[
\left(a_{m_{j}}r^{m_{j}}\right)^{2}\gtrsim\frac{H\left(t\right)}{\sqrt{B\left(t\right)}},
\]
and thus with probability at least $c>0$ we have that $\left|\xi_{m_{j}}a_{m_{j}}z^{m_{j}}\right|>\left|f\left(z\right)-\xi_{m_{j}}a_{m_{j}}z^{m_{j}}\right|$. 
\end{proof}

\section{Examples of Admissible Covariance Functions\label{sec:Examples}}

Here are some explicit examples for covarince functions which are
type I and type II admissible.

\subsection{The Mittag-Leffler function}

We consider the Gaussian entire function $f$ whose covariance function
$G$ is given by the Mittag-Leffler function 
\[
G(z)=G_{\alpha}(z)=\sum_{n\geq0}\frac{z^{n}}{\Gamma\left(1+\alpha^{-1}n\right)},
\]
where $\alpha>0$ is a parameter. Notice that $G_{1}(z)=e^{z}$ and
$G_{\tfrac{1}{2}}\left(z\right)=\cosh\sqrt{z}$. The asymptotic behavior
of $G$ and $G'$ is well-known, see for example \cite[Section 3.5.3]{goldberg2008value}.
In particular, as $|z|\to\infty$, and uniformly in $\arg z$,
\[
G(z)=\begin{cases}
\alpha e^{z^{\alpha}}+O\left(1\right), & \left|\arg z\right|\leq\frac{\pi}{2\alpha};\\
O\left(1\right), & \text{otherwise},
\end{cases}
\]
and 
\[
zG'(z)=\begin{cases}
\alpha^{2}z^{\alpha}e^{z^{\alpha}}+O\left(1\right), & \left|\arg z\right|\leq\frac{\pi}{2\alpha};\\
O\left(1\right), & \text{otherwise}.
\end{cases}
\]

\begin{rem}
Notice that for $\alpha\in\left(0,\tfrac{1}{2}\right]$ the complement
of $\left\{ \left|\arg z\right|\leq\frac{\pi}{2\alpha}\right\} $
is empty.
\end{rem}

From this asymptotic description, one easily verifies that $G$ is
type I and type II admissible. Since
\[
a\left(r\right)=\frac{rG^{\prime}\left(r\right)}{G\left(r\right)}\sim\alpha r^{\alpha},\qquad b\left(r\right)=ra^{\prime}\left(r\right)\sim\alpha^{2}r^{\alpha},\qquad r\to\infty,
\]
we have
\[
\mathbb{E}\left[n_{f}(r)\right]=a\left(r^{2}\right)\sim\alpha r^{2\alpha},\quad r\to\infty.
\]
By Theorem \ref{thm:var_asymp},
\[
\var{n_{f}\left(r\right)}\sim\frac{\zeta\left(\frac{3}{2}\right)}{4\sqrt{\pi}}\sqrt{b\left(r^{2}\right)}\sim\frac{\zeta\left(\frac{3}{2}\right)}{4\sqrt{\pi}}\cdot\alpha r^{\alpha},\quad r\to\infty.
\]

\subsection{\label{subsec:double_exp_example} The double exponent}

Here we consider the Gaussian entire function $f$ with covariance
function
\[
G(z)=e^{e^{z}}.
\]
The function $G$ is type I and type II admissible, and has an infinite
order of growth, with
\[
a\left(r\right)=re^{r},\qquad b\left(r\right)=r\left(r+1\right)e^{r}.
\]
Thus,
\[
\mathbb{E}\left[n_{f}(r)\right]=r^{2}e^{r^{2}},
\]
and by Theorem \ref{thm:var_asymp}
\[
\var{n_{f}\left(r\right)}\sim\frac{\zeta\left(\frac{3}{2}\right)}{4\sqrt{\pi}}r^{2}e^{\frac{1}{2}r^{2}},\quad r\to\infty.
\]

\subsection{The Lindelöf functions }

For $\alpha>0$, we consider the Gaussian entire function $f$ with
covariance function
\[
G(z)=G_{\alpha}\left(z\right)=\sum_{n\geq0}\frac{z^{n}}{\log^{\alpha n}\left(n+e\right)}.
\]
The function $G$ has infinite order of growth, and it follows from
\cite[Example 1.4.1]{KIROSODIN}, that it is type I and type II admissible
with, 
\[
\log G\left(r\right)\sim\frac{\alpha}{e}r^{-\frac{1}{\alpha}}\exp\left(r^{\frac{1}{\alpha}}\right),\quad r\to\infty,
\]
and
\[
a\left(r\right)\sim\frac{1}{e}\exp\left(r^{1/\alpha}\right),\qquad b\left(r\right)\sim\frac{r^{1/\alpha}}{\alpha e}\exp\left(r^{1/\alpha}\right),\qquad r\to\infty.
\]
In this case, 
\[
\mathbb{E}\left[n_{f}(r)\right]\sim e^{-1}e^{r^{2/\alpha}},\quad r\to\infty,
\]
\[
\var{n_{f}\left(r\right)}\sim\frac{\zeta\left(\frac{3}{2}\right)}{4\sqrt{\pi e\alpha}}\cdot r^{1/\alpha}e^{\frac{1}{2}r^{2/\alpha}},\quad r\to\infty.
\]

\subsection{An example with radius of convergence $1$\label{subsec:example_disk}}

For $\alpha>0$, we consider the Gaussian analytic function $f$,
where now the covariance function is given by 
\[
G(z)=\exp\left(\frac{1}{\left(1-z\right)^{\alpha}}\right).
\]
One can check that this function is type I admissible with $R_{G}=1$
and $C_{G}>2$ sufficiently large depending on $\alpha$. Since
\[
a\left(r\right)=\frac{\alpha r}{\left(1-r\right)^{\alpha+1}},\qquad b\left(r\right)=\frac{\alpha r}{\left(1-r\right)^{\alpha+1}}+\frac{\alpha\left(\alpha+1\right)r^{2}}{\left(1-r\right)^{\alpha+2}}.
\]
We have
\[
\mathbb{E}\left[n_{f}\left(r\right)\right]=\frac{\alpha r^{2}}{\left(1-r^{2}\right)^{\alpha+1}},\quad r<1,
\]
and Theorem \ref{thm:var_asymp} yields
\[
\var{n_{f}\left(r\right)}\sim\frac{\zeta\left(\frac{3}{2}\right)}{4\sqrt{\pi}}\frac{\sqrt{\alpha\left(\alpha+1\right)}}{\left(1-r^{2}\right)^{\tfrac{1}{2}\alpha+1}},\quad r\to1^{-}.
\]

\begin{rem}
For functions $G$ of \emph{slower} growth, the above asymptotics
\emph{no longer} holds. Buckley \cite{buckley2014fluctuations} found
the asymptotic of the variance for the following special choice
\[
G\left(z\right)=\frac{1}{\left(1-z\right)^{L}},\quad\text{with }L>0,
\]
which corresponds to Gaussian analytic functions whose zero sets are
\emph{invariant} with respect to the isometries of the hyperbolic
disk (see \cite[Chapter 2.3]{ZerosBook}). Earlier Peres and Virág
\cite{peres2005zeros} computed the variance in the case $L=1$, where
the zero set forms a \emph{determinantal} point process.
\end{rem}

\appendix

\section{\label{sec:Kahane_formulas} Formulas for the expectation and variance}

For the convenience of the reader, here we give proofs for the formulas
of the expected value and variance of the number of zeros in a disk
from \cite[p. 195]{Kahane}. Let $\Pi\subset\bbc$ be a compact subset
of the plane, and denote by $n_{f}\left(\Pi\right)$ the number of
zeros of the Gaussian analytic function $f$ in $\Pi$ (we assume
that $\Pi$ is contained inside the domain of convergence of $f$).
We denote by $K_{f}$ the covariance kernel of $f$, and by $J_{f}$
the normalized covariance kernel, given by
\[
J_{f}\left(z,w\right)=\frac{K_{f}\left(z,w\right)}{\sqrt{K_{f}\left(z,z\right)}\sqrt{K_{f}\left(w,w\right)}}.
\]

We first recall the Edelman-Kostlan formula \cite[p. 25]{ZerosBook},
which states
\begin{equation}
\ex{n_{f}\left(\Pi\right)}=\frac{1}{4\pi}\int_{\Pi}\Delta_{z}\log K_{f}\left(z,z\right)\,\dd m\left(z\right),\label{eq:Edelman-Kostlan}
\end{equation}
where $m$ is the Lebesgue measure on $\bbc$ and $\Delta_{z}=4\frac{\partial}{\partial z}\frac{\partial}{\partial\overline{z}}$
is the usual Laplace operator. In addition, we have (see \cite[Thm. 3.1]{shiffman2008number}
or \cite[Lemma 2.3]{nazarov2011fluctuations})
\begin{equation}
\var{n_{f}\left(\Pi\right)}=\tfrac{1}{16\pi^{2}}\iint_{\Pi\times\Pi}\Delta_{z}\Delta_{w}\mathrm{Li}_{2}\left(\left|J_{f}\left(z,w\right)\right|^{2}\right)\,\dd m\left(z\right)\,\dd m\left(w\right),\label{eq:variance_formula}
\end{equation}
where
\[
\mathrm{Li}_{2}\left(x\right)=\sum_{n=1}^{\infty}\frac{x^{n}}{n^{2}}
\]
is the dilogarithm function.

We now derive more explicit formulas when $\Pi=r\bbd\eqdef\left\{ \left|w\right|\le r\right\} $.
Recall that $n\left(r\right)\eqdef n_{f}\left(r\bbd\right)$, $K_{f}\left(z,w\right)=G\left(z\overline{w}\right)$,
\[
H\left(t\right)=G\left(e^{t}\right),\quad A\left(t\right)=H^{\prime}\left(t\right),\quad B\left(t\right)=A^{\prime}\left(t\right),
\]
and that we put $e^{t}=r^{2}$.
\begin{claim}
We have
\[
\mathbb{E}\left[n(r)\right]=a\left(r^{2}\right)=A\left(t\right).
\]
\end{claim}

\begin{proof}
Writing the Laplace operator in polar coordinates and differentiating
the covariance kernel, we get
\[
\Delta\log K_{f}(z,z)=\frac{1}{r}\frac{\partial}{\partial r}\left[r\frac{\partial}{\partial r}\right]\log G\left(r^{2}\right).
\]
Now, by (\ref{eq:Edelman-Kostlan}) we have, 
\[
\mathbb{E}\left[n\left(r\right)\right]=\frac{1}{2}\int_{0}^{r}\frac{\partial}{\partial s}\left[s\frac{\partial}{\partial s}\right]\log G\left(s^{2}\right)\,\dd s=\frac{1}{2}\left.\left[s\frac{\partial}{\partial s}\right]\log G\left(s^{2}\right)\right|_{s=r}=r^{2}\frac{G^{\prime}\left(r^{2}\right)}{G\left(r^{2}\right)}=a\left(r^{2}\right)=A\left(t\right).
\]
\end{proof}
\begin{claim}[{cf. \cite[Lemma 5]{buckley2014fluctuations}}]
\label{clm:var_formula_appendix} We have
\begin{align*}
\text{Var}\left(n_{f}\left(r\right)\right) & =\frac{1}{2\pi}\int_{-\pi}^{\pi}\frac{\left|G\left(r^{2}\right)G'\left(r^{2}e^{i\theta}\right)r^{2}e^{i\theta}-G\left(r^{2}e^{i\theta}\right)G'\left(r^{2}\right)r^{2}\right|^{2}}{G^{2}(r^{2})\left(G^{2}(r^{2})-\left|G^{2}(r^{2}e^{i\theta})\right|\right)}\,\dd\theta\\
 & =\frac{1}{2\pi}\int_{-\pi}^{\pi}\frac{\left|H\left(t\right)H^{\prime}\left(t+i\theta\right)-H\left(t+i\theta\right)H^{\prime}\left(t\right)\right|^{2}}{H^{2}\left(t\right)\left(H^{2}\left(t\right)-\left|H^{2}\left(t+i\theta\right)\right|\right)}\,\dd\theta\\
 & =\frac{1}{2\pi}\int_{-\pi}^{\pi}\frac{\left|A\left(t+i\theta\right)-A(t)\right|^{2}}{\exp\left(2\cdot\text{Im}\left[\int_{0}^{\theta}A(t+i\varphi)d\varphi\right]\right)-1}\,\dd\theta.
\end{align*}
\end{claim}

\begin{proof}
Applying Stokes\textquoteright{} Theorem to (\ref{eq:variance_formula})
we get
\begin{equation}
\text{Var}\left(n(r)\right)=-\tfrac{1}{4\pi^{2}}\ointop_{\partial\left(r\bbd\right)}\ointop_{\partial\left(r\bbd\right)}\frac{\partial}{\partial\overline{z}}\frac{\partial}{\partial\overline{w}}\mathrm{Li}_{2}\left(\left|J_{f}\left(z,w\right)\right|^{2}\right)\,\dd\overline{z}\,\dd\overline{w}.\label{eq:countour_variance_formula}
\end{equation}
Recall that
\[
\frac{\dd}{\dd\zeta}\mathrm{Li}_{2}\left(\zeta\right)=\frac{1}{\zeta}\log\frac{1}{1-\zeta},
\]
and therefore
\begin{align*}
\frac{\partial}{\partial\overline{w}}\mathrm{Li}_{2}\left(\left|J_{f}\left(z,w\right)\right|^{2}\right) & =\frac{\partial}{\partial\overline{w}}\mathrm{Li}_{2}\left(\frac{K_{f}\left(z,w\right)K_{f}\left(w,z\right)}{K_{f}\left(z,z\right)K_{f}\left(w,w\right)}\right)=\frac{\partial}{\partial\overline{w}}\mathrm{Li}_{2}\left(\frac{G\left(z\overline{w}\right)G\left(\overline{z}w\right)}{G\left(z\overline{z}\right)G\left(w\overline{w}\right)}\right)\\
 & =\log\left(1-\frac{G\left(z\overline{w}\right)G\left(\overline{z}w\right)}{G\left(z\overline{z}\right)G\left(w\overline{w}\right)}\right)\frac{\left(wG\left(z\overline{w}\right)G^{\prime}\left(w\overline{w}\right)-zG\left(w\overline{w}\right)G^{\prime}\left(z\overline{w}\right)\right)}{G\left(w\overline{w}\right)G\left(z\overline{w}\right)},
\end{align*}
hence, after some simplifications
\begin{multline*}
\frac{\partial}{\partial\overline{z}}\frac{\partial}{\partial\overline{w}}\mathrm{Li}_{2}\left(\left|J_{f}\left(z,w\right)\right|^{2}\right)=\\
\frac{\left(wG\left(z\overline{z}\right)G^{\prime}\left(\overline{z}w\right)-zG\left(\overline{z}w\right)G^{\prime}\left(z\overline{z}\right)\right)\left(zG\left(w\overline{w}\right)G^{\prime}\left(z\overline{w}\right)-wG\left(z\overline{w}\right)G^{\prime}\left(w\overline{w}\right)\right)}{G\left(z\overline{z}\right)G\left(w\overline{w}\right)\left[G\left(z\overline{z}\right)G\left(w\overline{w}\right)-G\left(z\overline{w}\right)G\left(\overline{z}w\right)\right]}.
\end{multline*}
Using the parametrization $z=re^{i\theta_{1}}$, $w=re^{i\theta_{2}}$
in (\ref{eq:countour_variance_formula}) (notice the contour $\partial\left(r\bbd\right)$
is oriented \emph{clockwise}) and after some additional simplifications,
we get, 
\[
\text{Var}\left(n_{f}\left(r\right)\right)=\tfrac{1}{4\pi^{2}}\int_{0}^{2\pi}\int_{0}^{2\pi}\frac{\left|G\left(r^{2}\right)G^{\prime}\left(r^{2}e^{i\left(\theta_{1}-\theta_{2}\right)}\right)r^{2}e^{i\left(\theta_{1}-\theta_{2}\right)}-G\left(r^{2}e^{i\left(\theta_{1}-\theta_{2}\right)}\right)G^{\prime}\left(r^{2}\right)r^{2}\right|^{2}}{G\left(r^{2}\right)^{2}\left[G\left(r^{2}\right)^{2}-\left|G\left(r^{2}e^{i\left(\theta_{1}-\theta_{2}\right)}\right)\right|^{2}\right]}\,\dd\theta_{1}\dd\theta_{2}.
\]
Making a change of variables $\theta=\theta_{1}-\theta_{2}$ and integrating
out the other variable, we get 
\[
\text{Var}\left(n_{f}\left(r\right)\right)=\frac{1}{2\pi}\int_{-\pi}^{\pi}\frac{\left|G\left(r^{2}\right)G'\left(r^{2}e^{i\theta}\right)r^{2}e^{i\theta}-G\left(r^{2}e^{i\theta}\right)G'\left(r^{2}\right)r^{2}\right|^{2}}{G^{2}(r^{2})\left(G^{2}(r^{2})-\left|G^{2}(r^{2}e^{i\theta})\right|\right)}\,\dd\theta.
\]
Now using $r^{2}=e^{t},$ we find that
\begin{align*}
\frac{\left|G\left(r^{2}\right)G'\left(r^{2}e^{i\theta}\right)r^{2}e^{i\theta}-G\left(r^{2}e^{i\theta}\right)G'\left(r^{2}\right)r^{2}\right|^{2}}{G^{2}(r^{2})\left(G^{2}(r^{2})-\left|G^{2}(r^{2}e^{i\theta})\right|\right)} & =\frac{\left|H\left(t\right)H^{\prime}\left(t+i\theta\right)-H\left(t+i\theta\right)H^{\prime}\left(t\right)\right|^{2}}{H^{2}\left(t\right)\left(H^{2}\left(t\right)-\left|H^{2}\left(t+i\theta\right)\right|\right)}\\
 & =\frac{\left|A\left(t+i\theta\right)-A(t)\right|^{2}}{\exp\left(-2\cdot\text{Re}\left[i\int_{0}^{\theta}A(t+i\varphi)d\varphi\right]\right)-1},
\end{align*}
and since $\text{Im}\left(z\right)=-\text{Re}\left(iz\right)$ we
get the required result.
\end{proof}
\bibliographystyle{amsplain}
\bibliography{bibRig}

\end{document}